\documentclass{amsart}
\usepackage{amsmath,amssymb,amscd,amsthm,amsfonts,amstext,amsbsy,mathrsfs,hyperref,upgreek,mathtools,stmaryrd,enumitem,bbm}

  \usepackage
  [a4paper, margin=1.6in]{geometry}

\usepackage[textsize=footnotesize,color=blue!40, bordercolor=white]{todonotes}
\hypersetup{colorlinks=true}
\numberwithin{equation}{section}
\usepackage[T1]{fontenc}

\newcommand{\RR}{\mathbb{R}}

\newcommand{\PP}{\mathbb{P}}
\newcommand{\QQ}{\mathbb{Q}}

\newcommand{\KK}{\mathbb{K}}
\newcommand{\HH}{\mathbb{H}}
\newcommand{\HHH}[1]{\mathrm{H}({#1})}

\newcommand{\pre}[2]{{}^{#1} #2}
\newcommand{\seq}[2]{\langle #1 \mid #2 \rangle}

\newcommand{\K}[1]{\mathrm{K}_{#1}}
\newcommand{\anf}[1]{{\text{``}\hspace{0.3ex}{#1}\hspace{0.3ex}\text{''}}}

\newcommand{\cof}{\operatorname{cof}}
\newcommand{\otp}[1]{{{\rm{otp}}\left(#1\right)}}
\newcommand{\supp}[1]{{{\rm{supp}}(#1)}}

\newcommand{\On}{{\mathrm{On}}}
\newcommand{\Lim}{{\mathrm{Lim}}}

\newcommand{\ZFC}{{\sf ZFC}}

\newcommand{\map}[3]{{#1} \colon {#2}\to{#3}}
\newcommand{\Map}[5]{{#1}\colon{#2}\to{#3} \colon{#4}\mapsto{#5}}
\newcommand{\pmap}[4]{{#1} \colon {#2}\xrightarrow{#4}{#3}}
\newcommand{\Set}[2]{\{{#1} \mid {#2}\}}
\newcommand{\lrSet}[2]{\left\{{#1} \mid {#2}\right\}}
\newcommand{\ran}[1]{{{\rm{ran}}(#1)}}
\newcommand{\dom}[1]{{{\rm{dom}}(#1)}}
\newcommand{\length}[1]{{{\rm{lh}}(#1)}}
\newcommand{\VV}{{\rm{V}}}
\newcommand{\WW}{{\rm{W}}}
\newcommand{\LL}{{\rm{L}}}
\newcommand{\Add}[2]{{\rm{Add}}({#1},{#2})}
\DeclarePairedDelimiterX{\goe}[1]{\prec}{\succ}{#1}
\newcommand{\goedel}[2]{\goe{#1, #2}}

\newtheorem{theorem}{Theorem}[section]
\newtheorem{lemma}[theorem]{Lemma}
\newtheorem{corollary}[theorem]{Corollary}
\newtheorem{proposition}[theorem]{Proposition}

\newtheorem{question}{Question}

\newtheorem*{claim*}{Claim}
\newtheorem*{subclaim*}{Subclaim}

\theoremstyle{definition}
\newtheorem{claim}{Claim}[theorem]
\newtheorem{definition}[theorem]{Definition}
\newtheorem{fact}[theorem]{Fact}

\theoremstyle{remark}

\newenvironment{enumerate-(a)}{\begin{enumerate}[label={\upshape (\alph*)}, leftmargin=2pc]}{\end{enumerate}}

\newenvironment{enumerate-(a)-r}{\begin{enumerate}[label={\upshape (\alph*)}, leftmargin=2pc,resume]}{\end{enumerate}}

\newenvironment{enumerate-(A)}{\begin{enumerate}[label={\upshape (\Alph*)}, leftmargin=2pc]}{\end{enumerate}}

\newenvironment{enumerate-(A)-r}{\begin{enumerate}[label={\upshape (\Alph*)}, leftmargin=2pc,resume]}{\end{enumerate}}

\newenvironment{enumerate-(i)}{\begin{enumerate}[label={\upshape (\roman*)}, leftmargin=2pc]}{\end{enumerate}}

\newenvironment{enumerate-(i)-r}{\begin{enumerate}[label={\upshape (\roman*)}, leftmargin=2pc,resume]}{\end{enumerate}}

\newenvironment{enumerate-(I)}{\begin{enumerate}[label={\upshape (\Roman*)}, leftmargin=2pc]}{\end{enumerate}}

\newenvironment{enumerate-(I)-r}{\begin{enumerate}[label={\upshape (\Roman*)}, leftmargin=2pc,resume]}{\end{enumerate}}

\newenvironment{enumerate-(1)}{\begin{enumerate}[label={\upshape (\arabic*)}, leftmargin=2pc]}{\end{enumerate}}

\newenvironment{enumerate-(1)-r}{\begin{enumerate}[label={\upshape (\arabic*)}, leftmargin=2pc,resume]}{\end{enumerate}}

\newenvironment{itemizenew}{\begin{itemize}[leftmargin=2pc]}{\end{itemize}}

\begin{document}

\author{Philipp L\"ucke} 
\address{Mathematisches Institut, Universit\"at Bonn,
Endenicher Allee 60, 53115 Bonn, Germany}
\email{pluecke@math.uni-bonn.de}

\author[L.~Motto Ros]{Luca Motto Ros}
\address{Dipartimento di matematica \guillemotleft{Giuseppe Peano}\guillemotright, Universit\`a di Torino, 
Via Carlo Alberto 10, 10121 Torino, Italy}
\email{luca.mottoros@unito.it}

 \author{Philipp Schlicht}
  \address{Institut f\"ur Mathematische Logik und Grundlagenforschung, Universit\"at M\"unster, Einsteinstr. 62, 48149 M\"unster, Germany and Mathematisches Institut,\
   Universit\"at Bonn\\Endenicher Allee 60\\53115 Bonn\\Germany} 
 \email{schlicht@math.uni-bonn.de}

\thanks{During the preparation of this paper, the first and the third author were partially supported by DFG-grant LU2020/1-1. The authors would like to thank the anonymous referee for the careful reading of the manuscript.}

\title[The Hurewicz dichotomy for generalized Baire spaces]{The Hurewicz dichotomy \\ for generalized Baire spaces} 

\begin{abstract}  
By classical results of Hurewicz, Kechris and Saint-Raymond, an analytic subset of a Polish space $X$ is covered by a  \( \K{\sigma} \) subset of $X$ if and only if it does not contain a closed-in-\( X \) subset homeomorphic to the Baire space ${}^{\omega}\omega$.  
We consider the analogous statement (which we call \emph{Hurewicz dichotomy}) for ${\bf \Sigma}^1_1$ subsets of the generalized Baire space ${}^{\kappa}\kappa$ for a given uncountable cardinal $\kappa$ with $\kappa=\kappa^{{<}\kappa}$. 
We show that the statement that this dichotomy holds at all uncountable regular cardinals is consistent with the axioms of $\ZFC$ together with $\mathsf{GCH}$  and large cardinal axioms. 
In contrast, we show that the dichotomy fails at all uncountable regular cardinals after we add a Cohen real to a model of $\mathsf{GCH}$. 
We also discuss connections with some regularity properties, like the \( \kappa \)-perfect set property, the $\kappa$-Miller measurability, and the $\kappa$-Sacks measurability. 
 
\end{abstract}

\maketitle

 \tableofcontents


\section{Introduction}\label{section:intro} 

A subset of a topological space \( X \)  is $\K{\sigma}$ if it can be written as a countable union of countably many compact subsets of \( X \).  
The space \( X \) is called $\K{\sigma}$ if it is a $\K{\sigma}$ subset of itself. It is easy to verify that if \( X \) is Hausdorff and \( A \subseteq X \) is \( \K{\sigma} \), then any closed subset of \( A \) is \( \K{\sigma} \) as well. A typical example of a non-\( \K{\sigma} \) Polish space is the Baire space ${}^\omega\omega$. This can be seen either by observing that a subset of \( \pre{\omega}{\omega} \) is  \( \K{\sigma} \) if and only if it is eventually bounded, or by applying the Baire Category Theorem. 
All together these observations show the following fact. 
\begin{fact} \label{fct:exclusiveomega}
A $\K{\sigma}$ subset of a Hausdorff topological space cannot contain a closed subset homeomorphic to ${}^\omega\omega$. 
\end{fact}
The following classical result in descriptive set theory due to Hurewicz (see {\cite[Theorem 7.10]{MR1321597}}) shows that the property in Fact~\ref{fct:exclusiveomega} actually characterizes the class of $\K{\sigma}$ Polish spaces.

\begin{theorem}[Hurewicz] \label{thm:hurewicz}
 A Polish space is $\K{\sigma}$ if and only if it does not contain a closed subset homeomorphic to ${}^{\omega}\omega$. 
\end{theorem} 

Fact~\ref{fct:exclusiveomega} together with~\cite[Theorems 3.11 and 4.17]{MR1321597} show that Theorem~\ref{thm:hurewicz} can be equivalently restated as: For any \( \mathbf{G}_\delta \) subset \( A \) of a Polish space \( X \), either \( A \) is contained in a \( \K{\sigma} \) subset of \( X \), or else \( A \) contains a closed-in-\( X \) homeomorphic copy of \( \pre{\omega}{\omega} \).
The following result of Kechris and Saint-Raymond (see {\cite[Corollary 21.23]{MR1321597}}) extends this dichotomy to analytic subsets of Polish spaces.

\begin{theorem}[Kechris, Saint-Raymond]\label{theorem:HDClassicAnalytic}
For any $\mathbf{\Sigma}^1_1$ subset \( A \) of a Polish space \( X \), either \( A \) is contained in a \( \K{\sigma} \) subset of \( X \), or else \( A \) contains a closed-in-\( X \) homeomorphic copy of \( \pre{\omega}{\omega} \).
\end{theorem}

\noindent
(The alternative in Theorem~\ref{theorem:HDClassicAnalytic} are again mutually exclusive by Fact~\ref{fct:exclusiveomega}.)
The results of~\cite{MR0450070} show that this dichotomy extends to all projective subsets of Polish spaces in  presence of the axiom of projective determinacy.

\medskip

In this paper, we consider the question of whether analogues of the above dichotomies hold if we replace $\omega$ with an uncountable regular cardinal $\kappa$ satisfying  $\kappa=\kappa^{{<}\kappa}$.  
For this purpose, we study the \emph{(\( \kappa \)-)generalized Baire space} ${}^\kappa\kappa$  consisting of all functions from $\kappa$ to $\kappa$ equipped with the \emph{bounded  topology}, i.e.\ the topology 
whose basic open sets are of the form $N_s = \Set{x\in{}^\kappa\kappa}{s\subseteq x}$ for some $s$ in the set ${}^{{<}\kappa}\kappa$ of all functions $\map{t}{\alpha}{\kappa}$ with $\alpha<\kappa$. When considering finite products \( (\pre{\kappa}{\kappa})^{n+1} \)  of the generalized Baire space, we always  endow them with the product of the bounded topology on \( \pre{\kappa}{\kappa} \).

We say that $T \subseteq ({}^{{<}\kappa}\kappa)^{n+1}$  is a \emph{subtree} (of $({}^{{<}\kappa}\kappa)^{n+1}$) if $\length{t_0}=\ldots=\length{t_n}$ and ${\langle {t_0\restriction\alpha},\dotsc,{t_n\restriction\alpha}\rangle\in T}$ for all $\langle t_0,\ldots,t_n\rangle\in T$ and $\alpha<\length{t_0}$. 
Given such a subtree $T$, we define $[T]$ (the \emph{body} of \( T \)) as the set of \emph{$\kappa$-branches through $T$}, i.e.\ the set of all $\langle x_0,\ldots,x_n\rangle\in({}^\kappa\kappa)^{n+1}$ such that $\langle x_0\restriction\alpha,\ldots,x_n\restriction\alpha\rangle\in T$ for all $\alpha<\kappa$.  
Note that a subset of $({}^\kappa\kappa)^{n+1}$ is closed with respect to the above topology if and only if it is of the form $[T]$ for some subtree $T \subseteq ({}^{{<}\kappa}\kappa)^{n+1}$. 
Given a closed  subset $C$ of $({}^\kappa\kappa)^{n+1}$, we call 
\begin{equation} \label{eq:canonicaltree}
 T_C = \Set{\langle x_0\restriction\alpha,\ldots,x_n\restriction\alpha\rangle}{\langle x_0,\ldots,x_n\rangle \in C , \, \alpha<\kappa}
\end{equation}
the \emph{(canonical) subtree induced by $C$}. 
We say that a subtree $T \subseteq ({}^\kappa\kappa)^{n+1}$ is \emph{pruned} if it is equal to the subtree induced by $[T]$. This is equivalent to require that every node in $T$ is extended by an element of $[T]$. 
Clearly the tree \( T_C \) is pruned for every closed subset $C$ of $({}^\kappa\kappa)^{n+1}$ and \( [T_C] = C \) holds for these subsets. 

Finally, we say that a subset of ${}^\kappa\kappa$ is  $\mathbf{\Sigma}^1_1$ if it is the projection of a closed subset of ${}^\kappa\kappa\times{}^\kappa\kappa$. 
Hence every $\mathbf{\Sigma}^1_1$ subset of \( \pre{\kappa}{\kappa} \) is the projection $p[T]$ of the set of all $\kappa$-branches through a (pruned) subtree \( T \) of ${}^{{<}\kappa}\kappa\times{}^{{<}\kappa}\kappa$. A subset of \( \pre{\kappa}{\kappa} \) is \( \mathbf{\Pi}^1_1 \) if its complement is \( \mathbf{\Sigma}^1_1 \).

The following  definition generalizes the notions of \( \K{\sigma} \) subset and  $\K{\sigma}$ space to our new setup.

\begin{definition} 
 Let $\kappa$ be an infinite cardinal and \( X \) be a topological space.
 \begin{enumerate-(i)} 
  \item A set $A \subseteq X$ is \emph{$\kappa$-compact} if every open cover of $A$ in $X$ has a subcover of size less than $\kappa$.  
  
  \item A set $A \subseteq X$ is $\K{\kappa}$ if it is a union of $\kappa$-many $\kappa$-compact subsets of $X$.

\item
The space \( X \) is \emph{\( \kappa \)-compact} (respectively, \( \K{\kappa} \)) if it is a \( \kappa \)-compact (respectively, \( \K{\kappa} \)) subset of itself.  
\end{enumerate-(i)} 
\end{definition}

Since the bounded topology on \( \pre{\kappa}{\kappa} \) is closed under intersections of size \( {<} \kappa \) whenever the cardinal $\kappa$ satisfies 
$\kappa=\kappa^{{<}\kappa}$, and since \( \pre{\kappa}{\kappa} \) is obviously a Hausdorff space, a standard argument shows that 
$\kappa$-compact subsets of ${}^\kappa\kappa$ are closed. Conversely, a closed subset of a \( \kappa \)-compact (respectively, \( \K{\kappa} \)) 
subset of \( \pre{\kappa}{\kappa} \) is \( \kappa \)-compact (respectively, \( \K{\kappa} \)). Finally, arguing as in the classical case \( \kappa = \omega \) 
it is not hard to check that the generalized Baire space \( \pre{\kappa}{\kappa} \), when \( \kappa \) is as above, is never \( \K{\kappa} \). (This is 
because although the converse may fail for some \( \kappa \), we still have that \( \K{\kappa} \) subsets of \( \pre{\kappa}{\kappa} \) are eventually 
bounded --- see Lemma~\ref{lemma:bounded}; alternatively, use the fact that the assumption \( \kappa^{< \kappa} = \kappa \) guarantees that  the 
space ${}^\kappa\kappa$ is $\kappa$-Baire, i.e.\ that ${}^\kappa\kappa$ is not a union of $\kappa$-many nowhere dense subsets of 
${}^\kappa\kappa$.) 
These observations show that (compare this with Fact~\ref{fct:exclusiveomega}):

\begin{fact} \label{proposition:DichotomyUnc}
A \( \K{\kappa} \) subset of \( \pre{\kappa}{\kappa} \) cannot contain a closed subset homeomorphic to \( \pre{\kappa}{\kappa} \).
\end{fact}

The results listed at the beginning of this introduction now motivate the following definition.

\begin{definition} \label{def:Hurewiczdichotomy}
 Given an infinite cardinal $\kappa$, we say that a set $A \subseteq {}^\kappa\kappa$ satisfies the \emph{Hurewicz dichotomy} 
 if either $A$ is contained in a $\K\kappa$ subset of ${}^\kappa\kappa$, or else $A$ contains a closed subset of ${}^\kappa\kappa$ homeomorphic to ${}^\kappa\kappa$. 
\end{definition}

Note that the two alternatives in the above definition are still mutually exclusive by Fact~\ref{proposition:DichotomyUnc}. 

Since $\K{\omega}$ is the same as $\K{\sigma}$, Theorem~\ref{theorem:HDClassicAnalytic} shows that $\mathbf{\Sigma}^1_1$ subsets of 
${}^\kappa \kappa$ satisfy the Hurewicz dichotomy if \( \kappa = \omega \).   The following result shows that it is possible to establish the Hurewicz 
dichotomy for $\mathbf{\Sigma}^1_1$ subsets of \( \pre{\kappa}{\kappa} \) also when \( \kappa^{< \kappa} = \kappa > \omega \) by forcing with a 
partial order that preserves many structural features of the ground model.

\begin{theorem}\label{theorem:Cons General}
 Let $\kappa$ be an uncountable cardinal with $\kappa=\kappa^{{<}\kappa}$. Then there is a partial order $\PP(\kappa)$ with the following properties.  
 \begin{enumerate-(1)}
  \item \label{theorem:Cons General-1}
$\PP(\kappa)$ is ${<}\kappa$-directed closed and $\kappa^+$-Knaster (hence it satisfies the \( \kappa^+ \)-chain condition).  
  \item \label{theorem:Cons General-2}
$\PP(\kappa)$ is a subset of $\HHH{\kappa^+}$ that is uniformly definable over $\langle\HHH{\kappa^+},\in\rangle$ in the parameter $\kappa$.\footnote{In the sense that the domain of $\PP(\kappa)$ and the relation $\leq_{\PP(\kappa)}$ are sets with this property.} 
  \item \label{theorem:Cons General-3}
$\PP(\kappa)$ forces that every $\mathbf{\Sigma}^1_1$ subset of ${}^\kappa\kappa$ satisfies the Hurewicz dichotomy. 
  \end{enumerate-(1)}
\end{theorem}

In the proof of this result, which is given in Section~\ref{sec:HDatagivenkappa}, we had to 
distinguish cases depending on whether $\kappa$ is weakly compact in the forcing extension or not. Such a distinction is connected to 
the observation that the spaces ${}^\kappa\kappa$ and ${}^\kappa 2$ are homeomorphic  if and only if $\kappa$ is not weakly compact 
(see {\cite[Theorem 1]{MR0367930}} and Corollary~\ref{corollary:RevNegroPonte}).

The methods developed in the non-weakly compact case actually allow us to prove a strengthening of the above result for such cardinals 
(see Section~\ref{sec:Hurewicznonweaklycompact} for the proof).

\begin{definition}\label{definition:StrongHD}
 Given an infinite cardinal $\kappa$, we say that  $A \subseteq {}^\kappa\kappa$ satisfies the \emph{strong Hurewicz dichotomy}%
\footnote{In the classical case \( \kappa = \omega \) this stronger version of the Hurewicz dichotomy holds for \( \mathbf{F}_\sigma \) sets, but there are \( \mathbf{G}_\delta \) counterexamples to it.}
 if either $A$ is a $\K\kappa$ subset of ${}^\kappa\kappa$, or else $A$ contains a closed subset of ${}^\kappa\kappa$ homeomorphic to ${}^\kappa\kappa$. 
\end{definition}

In this definition we are replacing the first alternative in Definition~\ref{def:Hurewiczdichotomy} with a stronger one.

\begin{theorem}\label{theorem:ConsNonWC}
 Let $\kappa$ be a non-weakly compact uncountable cardinal with $\kappa=\kappa^{{<}\kappa}$. Then there is a partial order $\PP$ with the following properties.  
 \begin{enumerate-(1)}
  \item $\PP$ is ${<}\kappa$-directed closed and $\kappa^+$-Knaster (hence it satisfies the \( \kappa^+ \)-chain condition). 

  \item $\PP$ forces that every $\mathbf{\Sigma}^1_1$ subset of ${}^\kappa\kappa$ satisfies the strong Hurewicz dichotomy. 
   \end{enumerate-(1)}
\end{theorem}

If instead $\kappa$ is a weakly compact cardinal, then it is easy to see that, as in the classical case \( \kappa = \omega \), there are always $\mathbf{\Sigma}^1_1$ (in fact, even \( \kappa \)-Borel%
\footnote{A subset of \( \pre{\kappa}{\kappa} \) is \emph{\( \kappa \)-Borel} if it belongs to the smallest \( \kappa^+ \)-algebra of subsets of \( \pre{\kappa}{\kappa} \) which contains all open sets.}%
) 
counterexamples to the strong Hurewicz dichotomy. Indeed,
by the results of~\cite{MR0367930} (see also Lemma~\ref{lemma:bounded}) the weak compactness of $\kappa$ implies that ${}^\kappa 2$ is a 
$\kappa$-compact subset of ${}^\kappa\kappa$, and therefore by Fact~\ref{proposition:DichotomyUnc} we have that every subset of ${}^\kappa 2$ 
that is not a union of $\kappa$-many closed subsets of ${}^\kappa 2$ does not satisfy the strengthened dichotomy. 
In particular, the set \( \{ x \in \pre{\kappa}{2} \mid \forall \alpha < \kappa \, \exists  \beta < \kappa \, ( \alpha < \beta \wedge x(\beta) = 1) \} \) is a \( \kappa \)-Borel subset of \( \pre{\kappa}{\kappa} \) that is homeomorphic to \( \pre{\kappa}{\kappa} \) and does not satisfy the strong Hurewicz dichotomy. 
Moreover, the results of~\cite{MR1880900} show that the club filter (viewed as a set of characteristic functions) is another natural example of a  
$\mathbf{\Sigma}^1_1$ subset of ${}^\kappa\kappa$ that does not satisfy the strong Hurewicz dichotomy. 
Nevertheless, when \( \kappa \) remains weakly compact after forcing with the partial order \( \PP(\kappa) \) from 
Theorem~\ref{theorem:Cons General}, we can strengthen our result in a different direction.

\begin{definition} \label{def:perfectandkappaMiller}
 Let $\nu \leq \kappa$ be  infinite cardinals with \( \kappa \) regular, and let $T$ be a subtree of ${}^{{<}\kappa}\kappa$.
 \begin{enumerate-(i)} 
  \item $T$ is \emph{superclosed} if it is closed under increasing sequences of length ${<}\kappa$. 
    
  \item $T$ is \emph{$\nu$-perfect} if it is superclosed and  for every $t\in T$ there is a node \( s \in T \) which extends $t$  and is $\nu$-splitting in $T$, i.e.\ it has at least \( \nu \)-many immediate successors.

  \item $T$ is \emph{perfect} if it is $2$-perfect, and it is a \emph{$\kappa$-Miller tree}%
\footnote{Note that $\omega$-Miller subtrees of \( \pre{< \omega}{\omega} \) are sometimes called \emph{superperfect trees} in the literature.}
 if it is $\kappa$-perfect.  
 \end{enumerate-(i)}
\end{definition}

Clearly every \( \kappa \)-Miller subtree of ${}^{{<}\kappa}\kappa$ is perfect, and every perfect subtree \( T \subseteq \pre{< \kappa}{\kappa} \) has height $\kappa$.
Moreover, every node in such a $T$ is extended by an element of $[T]$, and thus $T$ is pruned.

\begin{definition} \label{def:kappaMillerHurewiczdichotomy}
 Given an infinite cardinal $\kappa$, we say that a set $A \subseteq {}^\kappa\kappa$ satisfies the \emph{Miller-tree Hurewicz dichotomy} 
 if either $A$ is contained in a $\K\kappa$ subset of ${}^\kappa\kappa$, or else there is a $\kappa$-Miller tree $T$ with $[T]\subseteq A$. 
\end{definition}

Note that since Proposition~\ref{proposition:MillerTreeHomeoKappaKappa} shows that the set \( [T] \) is homeomorphic to \( \pre{\kappa}{\kappa} \) for every \( \kappa \)-Miller tree \( T \), the second alternative of Definition~\ref{def:kappaMillerHurewiczdichotomy} strengthens that of Definition~\ref{def:Hurewiczdichotomy}. 

In Theorem~\ref{theorem:MillTreeHD}
 we will show that if $\kappa$ is weakly compact, then $A \subseteq \pre{\kappa}{\kappa}$ satisfies the Hurewicz dichotomy if and only if it satisfies the Miller-tree Hurewicz dichotomy. Thus the following is an immediate corollary of Theorem~\ref{theorem:Cons General}.

\begin{theorem} \label{thm:weaklycompactintroduction}
Let \( \kappa \) be weakly compact. If \( \kappa \) remains weakly compact after forcing with the partial order \( \PP(\kappa) \) from Theorem~\ref{theorem:Cons General}, then in  the forcing extension every $\mathbf{\Sigma}^1_1$ subset of ${}^\kappa\kappa$ satisfies the Miller-tree Hurewicz dichotomy. 
\end{theorem}

Theorem~\ref{thm:weaklycompactintroduction} applies e.g.\ when $\kappa$ is a weakly compact cardinal whose weak compactness is indestructible with respect to ${<}\kappa$-closed forcings satisfying the \( \kappa^+ \)-chain condition; such a \( \kappa \) may be found e.g.\  in a suitable forcing extension of $\LL$, assuming the existence of a strongly unfoldable cardinal in $\LL$ (see~\cite{MR2467213}). 

 Notice also that if $\kappa$ is not weakly compact, then  the set ${}^{\kappa}2$ is a closed counterexample to the Miller-tree Hurewicz dichotomy: on the one hand it clearly cannot contain the body of a \( \kappa \)-Miller tree; on the other hand, it is a closed set homeomorphic to \( \pre{\kappa}{\kappa} \) (see 
Corollary~\ref{corollary:RevNegroPonte}), and thus it is not contained in any \( \K{\kappa} \) subset of \( \pre{\kappa}{\kappa} \) by Fact~\ref{proposition:DichotomyUnc}.

Conditions~\ref{theorem:Cons General-1} and~\ref{theorem:Cons General-2} of Theorem~\ref{theorem:Cons General} allow us to construct a class-sized forcing iteration (whose iterands are the forcings \( \PP(\kappa) \)) that forces the Hurewicz dichotomy for $\mathbf{\Sigma}^1_1$ sets to hold globally. 
In the light of the case distinction mentioned above, it is interesting to know whether such a class forcing preserves the weak compactness of certain large cardinals. 
The following theorem provides examples of such large cardinal notions. (Their definitions and the proof of the theorem can be found in Section~\ref{section:Global}.)

\begin{theorem}\label{theorem:GlobalHD}
 Assume that the $\mathsf{GCH}$ holds. Then there is a definable class forcing $\vec{\PP}$ with the following properties. 
 \begin{enumerate-(1)}
  \item \label{theorem:GlobalHD-1}
Forcing with $\vec{\PP}$ preserves all cofinalities, the $\mathsf{GCH}$, strongly unfoldable cardinals, and supercompact cardinals. 

  \item \label{theorem:GlobalHD-2}
 If $\kappa$ is an infinite regular cardinal, then $\vec{\PP}$ forces that every $\mathbf{\Sigma}^1_1$ subset of ${}^\kappa\kappa$ satisfies the Hurewicz dichotomy. 
 \end{enumerate-(1)}
\end{theorem}

Next we study the relationships between the Hurewicz dichotomy and some regularity properties of definable subsets of ${}^\kappa\kappa$. We start by considering the \( \kappa \)-perfect set property.

\begin{definition} 
 Given an infinite cardinal $\kappa$, we say that a set $A \subseteq {}^\kappa\kappa$ has the \emph{\( \kappa \)-perfect set property}  if either $A$ has cardinality at most $\kappa$, or else $A$ contains a closed subset of ${}^\kappa\kappa$ homeomorphic to ${}^\kappa2$. 
\end{definition}

In the case of non-weakly compact cardinals, the \( \kappa \)-perfect set property is a strengthening of the Hurewicz dichotomy.  
Moreover, the assumption that all closed subsets of ${}^\kappa\kappa$ have the \( \kappa \)-perfect set property allows us to fully characterize the class of $\K{\kappa}$ subsets of ${}^\kappa\kappa$ as the collection of sets of size \( {\leq} \kappa \) (see Proposition~\ref{prop:PSPHurewicz} and Corollary~\ref{cor:PSPHurewicz}).

An argument due to the third author (see e.g.\ {\cite[Proposition 9.9]{PL}}) shows that all $\mathbf{\Sigma}^1_1$ subsets of 
${}^\kappa\kappa$ have the \( \kappa \)-perfect set property after we Levy-collapsed an inaccessible cardinal to turn it into $\kappa^+$. (The models obtained in this way are sometimes called \emph{Silver models}, and are the natural generalizations of the Solovay model to uncountable cardinals.) Thus if 
\( \kappa \) turns out to be not weakly compact in the generic extension, then in such a Silver model  all \( \mathbf{\Sigma}^1_1 \) subsets of \( \pre{\kappa}{\kappa} \) satisfy 
the Hurewicz dichotomy as well. In Theorem~\ref{theorem:HDinCollapse} we will actually show that the same is true also when $\kappa$ is a weakly 
compact cardinal (in the generic extension). 

The Silver model cannot in general be used to separate the various regularity properties: in fact, the results from~\cite{PS,Lag}  show that many of these properties are simultaneously  satisfied by \( \mathbf{\Sigma}^1_1 \) subsets of \( \pre{\kappa}{\kappa} \). However, 
we can use Theorem~\ref{theorem:GlobalHD} to separate the Hurewicz dichotomy from the \( \kappa \)-perfect set property. This is because 
if we force over \( \LL \) with the class forcing \( \PP \) from Theorem~\ref{theorem:GlobalHD}, then in the generic extension we get that for any 
uncountable regular cardinal \( \kappa \) there is a subtree $T \subseteq {}^{{<}\kappa}\kappa$ (which actually belongs to \( \LL \)) such that  the 
closed set $[T]$ does not have the \( \kappa \)-perfect set property (see Theorem~\ref{theorem:SeparatePSPfromHD}). Combining this with the 
mentioned Theorem~\ref{theorem:GlobalHD}, we get a model of \( \ZFC + \mathsf{GCH} \) in which all \( \mathbf{\Sigma}^1_1 \) subsets of 
\( \pre{\kappa}{\kappa} \) (\( \kappa \) an arbitrary regular uncountable cardinal) satisfy the Hurewicz dichotomy while the \( \kappa \)-perfect set 
property fails already for closed sets (Corollary~\ref{cor:HDnonPSP}).

We also consider other two regularity properties which are natural generalizations of the notions of \emph{Miller measurability} and \emph{Sacks measurability} from $\omega$ to uncountable regular cardinals.%
\footnote{The choice for such generalizations is not unique: other variants have been considered e.g.\ in~\cite{Lag,KKF}.}

\begin{definition} \label{def:regularityproperties}
 Let $\kappa$ be an infinite regular cardinal and $A$ be a subset of ${}^\kappa\kappa$. We say that $A$ is \emph{$\kappa$-Sacks measurable} (respectively, \emph{$\kappa$-Miller measurable}) if for every perfect (respectively, \( \kappa \)-Miller) subtree $T \subseteq {}^{{<}\kappa}\kappa$ there is a perfect (respectively, \( \kappa \)-Miller) subtree $S \subseteq T$ such that either $[S]\subseteq A$ or $[S]\cap A=\emptyset$. 
\end{definition}

Notice that \( \kappa \)-Sacks measurability is essentially a symmetric version of the \( \kappa \)-perfect set property, and in fact it is easy to see that if a set \( A \subseteq \pre{\kappa}{\kappa} \) has the \( \kappa \)-perfect set property, then it is also \( \kappa \)-Sacks measurable. 
Thus in a Silver model all \( \mathbf{\Sigma}^1_1 \) subsets (hence also all \(  \mathbf{\Pi}^1_1 \) subsets) of \( \pre{\kappa}{\kappa} \) are \( \kappa \)-Sacks measurable. Laguzzi showed in~\cite{Lag}  that in fact such sets are also \( \kappa \)-Miller measurable. 

In view of these results, it is  natural to ask whether it is possible to separate the \( \kappa \)-Miller measurability and the \( \kappa \)-Sacks 
measurability from the \( \kappa \)-perfect set property. In Theorem~\ref{theorem:HDimpliesMeasurability} we will show that the fact that all 
\( \mathbf{\Sigma}^1_1 \) sets satisfy the Hurewicz dichotomy directly implies that all these sets, together with their complements, are 
\( \kappa \)-Miller measurable or \( \kappa \)-Sacks measurable (which alternative applies depends again on whether \( \kappa \) is weakly compact or not). 
Combining this result with our Theorem~\ref{theorem:GlobalHD} and some of the previous observations, we will thus get that it is consistent to have 
a model in which e.g.\ all \( \mathbf{\Sigma}^1_1 \) and all \( \mathbf{\Pi}^1_1 \) sets are \( \kappa \)-Sacks measurable (alternatively, \( \kappa \)-Miller measurable) but there is a 
closed set which does not satisfy the \( \kappa \)-perfect set property (Corollary~\ref{cor:measurabilitynonPSP}). Thus the regularity properties from Definition~\ref{def:regularityproperties} are really weaker than the \( \kappa \)-perfect set property.

Finally, we will also consider failures of the Hurewicz dichotomy and show that the construction of counterexamples to the dichotomy leads to 
interesting combinatorial concepts and results. 
In Theorem~\ref{theorem:CounterexampleCohenReal} we will show that adding an element of \( \pre{\mu}{\mu} \) with the usual 
\( \mu \)-Cohen forcing (for \(\mu\) an infinite cardinal satisfying \( \mu^{< \mu} = \mu \)) provides a model in which for all \( \kappa > \mu \) such 
that \( \kappa^{<\kappa} = \kappa \) there are closed subsets of \( \pre{\kappa}{\kappa} \) which do not satisfy the Hurewicz dichotomy. 
In particular, starting from a model of the \( \mathsf{GCH} \) and taking \( \mu = \omega \) in the above construction we obtain that it is consistent that there are closed 
counterexamples to the Hurewicz dichotomy at every uncountable regular cardinal \( \kappa \). 
In addition, we will show that also in G\"odel's constructible universe $\LL$  there are closed counterexamples to the Hurewicz dichotomy at every uncountable regular cardinal (see Theorem~\ref{theorem:CounterexampleInL}). 
These closed sets are obtained using a modification of the construction of a $\kappa$-Baire almost $\kappa$-Kurepa tree due to Friedman and Kulikov (see~\cite{FriedmanKulikov}).

In the remainder of this paper, we will work in \( \mathsf{ZFC} \) and let $\kappa$ denote an uncountable cardinal with $\kappa = \kappa^{{<}\kappa}$. 
Note that this assumption implies that $\kappa$ is regular.


\section{Basic combinatorial results}

We start by recalling some definitions and establishing some basic results.

\begin{definition}\label{def:kappa-tree}
 Let $T$ be a subtree of ${}^{{<}\kappa}\kappa$. 
  \begin{enumerate-(i)} 
  \item \emph{$T$ has height $\kappa$} if $T\cap{}^\alpha\kappa\neq\emptyset$ for all $\alpha<\kappa$. 

  \item $T$ is a \emph{$\kappa$-tree} if it has height $\kappa$ and  for every $\alpha<\kappa$ the set $T\cap{}^\alpha\kappa$ has cardinality less than $\kappa$. 
 
  \item $T$ is a \emph{$\kappa$-Aronszajn tree} if it is a $\kappa$-tree and $[T]=\emptyset$. 
  
  \item $T$ is a \emph{$\kappa$-Kurepa tree} if it is a $\kappa$-tree and that set $[T]$ has cardinality greater than $\kappa$.
\end{enumerate-(i)} 
\end{definition}

The following lemma provides a combinatorial characterization of $\kappa$-compact subsets of ${}^\kappa\kappa$.  
The lemma was proved for the case \anf{$\kappa=\omega_1$} in {\cite[Theorem 3]{MR509548}} and the presented proof directly 
generalizes to higher cardinalities. 
  For sake of completeness, we include it here.

\begin{lemma}\label{lemma:kurepa}
Given a pruned subtree $T \subseteq {}^{{<}\kappa}\kappa$, the following statements are equivalent.  
 \begin{enumerate-(1)}
  \item $[T]$ is $\kappa$-compact.  
  
  \item $T$ is a $\kappa$-tree without $\kappa$-Aronszajn subtrees (i.e.\ there is no subtree $S$ of ${}^{{<}\kappa}\kappa$ such that $S\subseteq T$ and $S$ is a $\kappa$-Aronszajn tree).   
 \end{enumerate-(1)}
In particular, a closed \( C \subseteq \pre{\kappa}{\kappa} \) is \( \kappa \)-compact if and only it the canonical subtree induced by \( C \) defined in~\eqref{eq:canonicaltree} is a \( \kappa \)-tree without \( \kappa \)-Aronszajn subtrees.
\end{lemma} 

\begin{proof} 
 Suppose that $[T]$ is $\kappa$-compact. 
 Assume towards a contradiction that there is an $\alpha<\kappa$ such that $T\cap{}^\alpha\kappa$ has cardinality $\kappa$. 
 Then $\Set{N_t}{t\in T\cap{}^\alpha\kappa}$ is a cover of $[T]$ without a subcover of cardinality less than $\kappa$, a contradiction.  
 Now assume towards a contradiction that $T$ contains a $\kappa$-Aronszajn subtree $S$. Define 
 \begin{equation}\label{equation:BoundaryTree}
  \partial S  = \Set{t\in T\setminus S}{\forall \alpha<\length{t} \, (t\restriction\alpha\in S)}.
 \end{equation} 
 If $x\in[T]$, then since $x\notin[S] = \emptyset$ there is an $\alpha<\kappa$ with $x\restriction\alpha\in\partial S$. 
 This shows that $\mathcal{C} =  \Set{N_t}{t\in\partial S}$ is a disjoint open cover of $[T]$. 
 Then \( \mathcal{C} \) has cardinality \( \kappa \), because \( S \) has height \( \kappa \). 
 Since \( \mathcal{C} \) is a partition, it cannot be refined to any proper subcover and this contradicts the fact that \( [T] \) is \( \kappa \)-compact.

 For the other direction, suppose that $T$ is a $\kappa$-tree without $\kappa$-Aronszajn subtrees. Let $\mathcal{U}$ be an open cover of $[T]$. Define  
 \begin{equation*}
  S  =  \Set{t\in T}{\text{$\mathcal{U}$ does not have a subcover of $N_t\cap[T]$ of cardinality less than $\kappa$}}.
 \end{equation*}
 Then $S$ is a subtree of ${}^{{<}\kappa}\kappa$ and it is a \( \kappa \)-tree, because \( S \subseteq T \). 
 We first show that \( S \subseteq \pre{< \alpha}{\kappa} \) for some \( \alpha < \kappa \). Assume towards a contradiction that \( S \) has height \( \kappa \). 
 By our assumption, \( S \) is not a \( \kappa \)-Aronszajn tree and hence \( [S] \neq \emptyset \). 
 Fix any \( x \in [S] \). Since \( \mathcal{U} \) is an open cover of \( [T] \), there are $U\in\mathcal{U}$ and  $\alpha<\kappa$ such that  $N_{x\restriction\alpha}\subseteq U$. This implies \( x \restriction \alpha \notin S \), a contradiction.
Given now \( \alpha < \kappa \) such that \( S \subseteq \pre{< \alpha}{\kappa} \), for every  \( t \in T \cap \pre{\alpha}{\kappa} \subseteq T \setminus S \) there is a subcover $\mathcal{U}_t \subseteq \mathcal{U}$ of $N_t\cap[T]$ of cardinality less than $\kappa$.
 Since $T\cap{}^\alpha\kappa$ has cardinality less than $\kappa$ and \( \kappa \) is regular, $\bigcup\Set{\mathcal{U}_t}{t\in T\cap{}^\alpha\kappa} \subseteq \mathcal{U}$ is a subcover of $[T]$ of cardinality less than $\kappa$.  
\end{proof}

The above lemma shows that, as pointed out in~\cite[Theorem 5.6]{MR3093397}, the space ${}^{\kappa}2$ is $\kappa$-compact if and only if $\kappa$ is weakly compact.  
This implies that the implication in {\cite[Theorem 1]{MR0367930}} can be reversed.

\begin{corollary}\label{corollary:RevNegroPonte}
 The spaces ${}^\kappa 2$ and ${}^\kappa\kappa$ are homeomorphic if and only if $\kappa$ is not weakly compact.  
\end{corollary}

Given $x,y\in{}^\kappa\kappa$, we write $x\leq y$ if $x(\alpha)\leq y(\alpha)$ for all $\alpha<\kappa$ and $x\leq^* y$ if there is a $\beta<\kappa$ such that $x(\alpha)\leq y(\alpha)$ for all $\beta\leq\alpha<\kappa$.

\begin{definition} \label{def:bounded} 
 Let $A$ be a subset of ${}^\kappa\kappa$.  
 \begin{enumerate-(i)}
  \item We say that $A$ is \emph{bounded} if there is an $x\in{}^\kappa\kappa$ with $y\leq x$ for every $y\in A$. 
  
  \item We say that $A$ is \emph{eventually bounded} if there an $x\in{}^\kappa\kappa$ such that $y \leq^* x$ for all $y \in A$. 
\end{enumerate-(i)}
\end{definition}

\begin{proposition}\label{proposition:bounded}
 If $A=\bigcup_{\alpha < \kappa} A_\alpha \subseteq {}^\kappa\kappa$ and each $A_\alpha$ is bounded, 
 then $A$ is eventually bounded. 
\end{proposition}

\begin{proof}
 Given $\alpha<\kappa$, pick $x_\alpha\in{}^\kappa\kappa$ with $y\leq x_\alpha$ for all $y \in A_\alpha$. 
 Define $x\in{}^\kappa\kappa$ by setting $x(\alpha) = \sup_{\beta \leq \alpha} x_\beta(\alpha)$ for all $\alpha<\kappa$: 
 then $x$ witnesses that $A$ is eventually bounded.  
\end{proof}

\begin{lemma}\label{lemma:bounded}
 Let \( A \) be a subset of \( \pre{\kappa}{\kappa} \).
 \begin{enumerate-(1)}
  \item \label{prop:bounded-01} If $A$ is contained in a $\kappa$-compact subset of \( \pre{\kappa}{\kappa} \), then $A$ is bounded. 
  
  \item \label{prop:bounded-02} If $A$ is contained in a $\K{\kappa}$ subset of \( \pre{\kappa}{\kappa} \), then $A$ is eventually bounded. 
    
  \item \label{prop:bounded-03} If $\kappa$ is weakly compact, then both the implications above can be reversed: The set $A$ is contained in a $\kappa$-compact (respectively, \( \K{\kappa} \)) subset if and only if $A$ is bounded (respectively, eventually bounded).
\end{enumerate-(1)}
\end{lemma}

\begin{proof}
 Assume first that $A$ is contained in a $\kappa$-compact subset. 
 By Lemma~\ref{lemma:kurepa}, the canonical subtree \( T \subseteq{}^{{< }\kappa}\kappa \) induced by the closure of \( A \) is a $\kappa$-tree with $A \subseteq [T]$. 
 Given  $\alpha < \kappa$, we define $L(\alpha) = \Set{ t(\alpha)}{t \in T, \, \alpha\in\dom{t}}$. 
 Then $L(\alpha)$ is a bounded subset of $\kappa$, because $T\cap{}^{\alpha+1}\kappa$ has cardinality less than $\kappa$.  
 Define $x\in{}^\kappa\kappa$ by setting \( x(\alpha) = \sup L(\alpha) \) for all \( \alpha < \kappa \). Then $x$ witnesses that $A$ is bounded.  
 This proves~\ref{prop:bounded-01} and, by Proposition~\ref{proposition:bounded}, this argument also yields~\ref{prop:bounded-02}.
 
To prove~\ref{prop:bounded-03}, assume that $\kappa$ is weakly compact. 
 Let $x\in{}^\kappa\kappa$ witness that $A$ is bounded, and define 
 \begin{equation*}
  T  =  \Set{t \in {}^{{<}\kappa}\kappa }{\forall\alpha\in\dom{t} \, ( t(\alpha) \leq x(\alpha))}.
 \end{equation*}  
 Then $T$ is a pruned subtree of  ${}^{{<}\kappa}\kappa$ with $A\subseteq[T]$. Moreover, the inaccessibility of $\kappa$ implies that $T$ is a $\kappa$-tree and this implies that it has no $\kappa$-Aronszajn subtrees, because $\kappa$ has the tree property. 
 By Lemma~\ref{lemma:kurepa}, this shows that $[T]$ is $\kappa$-compact. 
 If $A$ is eventually bounded, then $A$ is contained in a union of $\kappa$-many bounded set, and thus the above computations show that $A$ is contained in a $\K{\kappa}$ subset.  
\end{proof}

We now establish some connections between the \( \kappa \)-perfect set property, \( \K{\kappa} \) sets, and the Hurewicz dichotomy. 

\begin{proposition} \label{prop:PSPHurewicz}
 Assume that $\kappa$ is not weakly compact, and let $A$ be a subset of ${}^\kappa\kappa$ with the \( \kappa \)-perfect set property. Then the following statements are equivalent. 
 \begin{enumerate-(1)}
   \item The set $A$ contains a closed subset homeomorphic to ${}^\kappa\kappa$.
 
  \item The set $A$ is not contained in a $\K{\kappa}$ subset of ${}^\kappa\kappa$. 
  
  \item The set $A$ has cardinality greater than $\kappa$.  
 \end{enumerate-(1)}
In particular, the set \( A \) satisfies the strong Hurewicz dichotomy.
\end{proposition} 

\begin{proof}
 By Fact~\ref{proposition:DichotomyUnc}, no closed subset homeomorphic to ${}^\kappa\kappa$  is  contained in a $\K{\kappa}$ subset of ${}^\kappa\kappa$.   
 Clearly, every subset of ${}^\kappa\kappa$ of cardinality at most $\kappa$ is a $\K{\kappa}$ subset. 
 Assume that  $A$ has cardinality greater than $\kappa$. Then by the \( \kappa \)-perfect set property \( A \) contains a closed subset homeomorphic to \( \pre{\kappa}{2} \). 
 Since under the above assumptions  the spaces ${}^\kappa\kappa$ and ${}^\kappa 2$ are homeomorphic by {\cite[Theorem 1]{MR0367930}}, we can conclude that $A$ contains a closed subset homeomorphic to ${}^\kappa\kappa$.  
\end{proof}

As a consequence, we get that the \( \kappa \)-perfect set property for closed sets allows us to fully characterize \( \K{\kappa} \) sets in the non-weakly compact case.

\begin{corollary} \label{cor:PSPHurewicz}
 Assume that $\kappa$ is not weakly compact and  that all closed subsets of ${}^\kappa\kappa$ have the \( \kappa \)-perfect set property. Then for every \( A \subseteq \pre{\kappa}{\kappa} \), \( A \) is (contained in) a \( \K{\kappa} \) set if and only if $A$ has cardinality at most $\kappa$. 
\end{corollary}

\begin{proof}
For the nontrivial direction, let \( K_\alpha \subseteq \pre{\kappa}{\kappa} \), \( \alpha < \kappa \), be \( \kappa \)-compact sets with \( A \subseteq \bigcup_{\alpha < \kappa} K_\alpha \). By our assumption, each \( K_\alpha \), being closed,  has the \( \kappa \)-perfect set property, and hence has size \( {\leq} \kappa \) by Proposition~\ref{prop:PSPHurewicz}. It follows that \( A \) has size \( {\leq} \kappa \) as well.
\end{proof}

Finally, we provide combinatorial reformulations of the properties of containing a closed subset homeomorphic to \( \pre{\kappa}{2} \) or \( \pre{\kappa}{\kappa} \). Recall from Definition~\ref{def:perfectandkappaMiller} the notions of a perfect and \( \kappa \)-Miller tree. Given an inclusion-preserving function $\map{\iota}{{}^{{<}\kappa}\kappa}{{}^{{<}\kappa}\kappa}$, we say that $\iota$ is \emph{continuous}  if $\iota(u)=\bigcup\Set{\iota(u\restriction\alpha)}{\alpha<\length{u}}$ holds for all  $u\in{}^{{<}\kappa}\kappa$ with $\length{u}\in\Lim$.

\begin{lemma}\label{lemma:PerfectClosedTreeEmb}
 Given a subtree $T \subseteq \pre{< \kappa}{\kappa}$, the following statements are equivalent:
 \begin{enumerate-(1)}
  \item \label{item:homeo} the set $[T]$ contains a closed subset homeomorphic to ${}^\kappa 2$;

  \item\label{item:ContInjection} there is a continuous injection $\map{i}{{}^\kappa 2}{{}^\kappa \kappa}$ with $\ran{i}\subseteq[T]$;

  \item\label{item:TreeInjection} there is an inclusion-preserving continuous injection $\map{\iota}{{}^{{<}\kappa}2}{T}$; 
 
  \item \label{item:perfect} the tree $T$ contains a perfect subtree. 
 \end{enumerate-(1)}
\end{lemma}

\begin{proof}
The implications~\ref{item:homeo} \( \Rightarrow \)~\ref{item:ContInjection},~\ref{item:TreeInjection} \( \Rightarrow \)~\ref{item:perfect}, and~\ref{item:perfect} \( \Rightarrow \)~\ref{item:homeo} are obvious, so let us prove~\ref{item:ContInjection} \( \Rightarrow \)~\ref{item:TreeInjection}. 
 Let $\map{i}{{}^\kappa 2}{{}^\kappa \kappa}$ be a continuous injection with $\ran{i}\subseteq[T]$.  
 We will recursively construct inclusion-preserving continuous injections $\map{e}{{}^{{<}\kappa}2}{{}^{{<}\kappa}2}$ and $\map{\iota}{{}^{{<}\kappa}2}{T}$ 
 such that the following statements hold for all  $s\in{}^{{<}\kappa}2$:  
 \begin{enumerate-(a)}
  \item \label{item:e} $i[N_{e(s)}\cap \pre{\kappa}{2}] \subseteq N_{\iota(s)} \cap [T]$. 
  
  \item \label{item:iota} $N_{\iota(s {}^\smallfrown{}   0 )}\cap N_{\iota(s {}^\smallfrown{}  1)}=\emptyset$. 
 \end{enumerate-(a)}
The auxiliary map \( e \) and the corresponding condition~\ref{item:e} will ensure that our construction of \( \iota \) can go through limit levels maintaining the property that its range is contained in \( T \) (which cannot in general be assumed to be superclosed). The map \( \iota \) will thus be as in~\ref{item:TreeInjection} of the lemma.

  Assume that we already defined $\iota(s)$ for some $s\in{}^{{<}\kappa}2$. Pick $x_0,x_1\in N_{e(s)}\cap{}^\kappa 2$ with $x_0\neq x_1$. 
  Then $i(x_0)\neq i(x_1)$ and $i(x_0),i(x_1)\in N_{\iota(s)}$ by~\ref{item:e}. 
Hence by continuity of \( i \)  we can find ordinals $\length{\iota(s)}<\alpha<\kappa$ and $\length{e(s)}<\beta<\kappa$ 
  such that $i(x_0)\restriction\alpha\neq i(x_1)\restriction\alpha$ and $i[N_{x_j\restriction\beta}\cap{}^\kappa 2]\subseteq N_{i(x_j)\restriction\alpha} \cap [T]$ for \( j = 0,1 \). 
  It is then enough to set $e(s^\smallfrown  j  )=x_j \restriction\beta$ and $\iota(s^\smallfrown  j )=i(x_j)\restriction\alpha$.   
\end{proof}

The following result used in the proof of Theorem~\ref{theorem:CounterexampleCohenReal} follows directly from the above proof (condition~\ref{item:e}  guarantees that \( S =  \ran{\iota} \) is such that \( [S] \subseteq \ran{i} \)) and the observation that the same argument works for subtrees of ${}^{{<}\omega}\omega$.

\begin{corollary}\label{corollary:PerfectClosedTreeEmb}
 If  $\mu$ is an infinite cardinal with $\mu=\mu^{{<}\mu}$, $T$ is a subtree  of ${}^{{<}\mu}\mu$ and $\map{i}{{}^\mu 2}{{}^\mu \mu}$ is a continuous injection with $\ran{i}\subseteq[T]$, 
 then there is a perfect subtree $S \subseteq T$ with $[S]\subseteq\ran{i}$. 
\end{corollary}

The next proposition provides the analogue of Lemma~\ref{lemma:PerfectClosedTreeEmb} for the case of closed homeomorphic copies of \( \pre{\kappa}{\kappa} \). 

\begin{proposition}\label{proposition:MillerTreeHomeoKappaKappa}
Let 
 \( T \subseteq \pre{< \kappa}{\kappa} \) be a pruned subtree. 
\begin{enumerate-(1)}
\item \label{item:kappaMiller}
If \( S \subseteq T \) is a \( \kappa \)-Miller tree, then \( [S] \) is a closed set homeomorphic to \( \pre{\kappa}{\kappa} \). 
\item \label{item:weaklycompact}
Assume that \( \kappa \) is weakly compact. If \( [T] \) contains a closed set homeomorphic to \( \pre{\kappa}{\kappa} \),  then there is a $\kappa$-Miller subtree $S \subseteq T$.
\end{enumerate-(1)}
\end{proposition}

\begin{proof}
For part~\ref{item:kappaMiller}, notice that the splitting properties of $S$ allow us to inductively construct an inclusion-preserving continuous injection $\map{\iota}{{}^{{<}\kappa}\kappa}{S}$ such that 
 \[
\Set{\iota(u^\smallfrown \alpha )}{\alpha<\kappa}  =  \Set{s\in S}{\iota(u)\subseteq s,  \length{s}=\length{\iota(u^\smallfrown  0 )}}
\]
 holds for all $u\in{}^{{<}\kappa}\kappa$.  This injection allows us to define a homeomorphism 
 \begin{equation*}
  \Map{f}{{}^\kappa\kappa}{[S]}{x}{\bigcup\Set{\iota(x\restriction\alpha)}{\alpha<\kappa}}. 
 \end{equation*}

We now prove part~\ref{item:weaklycompact}. Without loss of generality, we may assume that \( [T] \) itself be homeomorphic to \( \pre{\kappa}{\kappa} \), and let $\map{f}{{}^{\kappa}\kappa}{[T]}$ be a homeomorphism.

 \begin{claim} \label{claim:meet}
  If $\bar{s}\in{}^{{<}\kappa}\kappa$ and $\bar{t}\in[T]$ with $N_{\bar{t}}\subseteq f[N_{\bar{s}}]$, then there are $s\in{}^{{<}\kappa}\kappa$ and $t\in T$ such that $\bar{s}\subseteq s$, $\bar{t}\subseteq t$ and $f[N_s]=N_t\cap[T]$. 
 \end{claim}

 \begin{proof}[Proof of the Claim]
  Pick $x\in N_{\bar{s}}$ with $f(x)\in N_{\bar{t}}$. Using the fact that \( f \) is a homeomorphism, inductively define a strictly increasing sequence $\seq{\alpha_n<\kappa}{n<\omega}$ of ordinals such that 
  $\alpha_0=\max\{\length{\bar{s}},  \length{\bar{t}}\}$ and 
 \[
{N_{f(x)\restriction\alpha_{2n+2}}\cap[T]}  \subseteq  {f[N_{x\restriction\alpha_{2n+1}}]}  \subseteq  {N_{f(x)\restriction\alpha_{2n}}}
\]
  for  all $n<\omega$. Set $\alpha=\sup\Set{\alpha_n}{n<\omega}$, $s=x\restriction\alpha$, and $t=f(x)\restriction\alpha$. Using the fact that \( f \) is a bijection, 
 we get that  \( s \) and \( t \) are as required..  
 \end{proof}

 \begin{claim} \label{claim:kappasplitting}
  If $t\in T$, then there is a $\kappa$-splitting node $s$ in $T$ with $t\subseteq s$.
 \end{claim}

 \begin{proof}[Proof of the Claim]
  Assume, towards a contradiction, that there is a $t\in T$ without $\kappa$-splitting nodes above it. Let $U$ be a pruned subtree of $T$ with $[U]=N_t\cap[T]$. 
  Since \( \kappa \) is inaccessible,  $U$ is a $\kappa$-tree. By Lemma~\ref{lemma:kurepa}, the weak compactness of $\kappa$ implies that $[U]$ is $\kappa$-compact in ${}^\kappa\kappa$ and hence also in $[T]$. 
  Thus  the preimage of $[U]$ under the homeomorphism $f$ would be $\kappa$-compact in ${}^\kappa\kappa$ and would be open nonempty, a contradiction. 
 \end{proof}

Claims~\ref{claim:meet} and~\ref{claim:kappasplitting} allow us to recursively construct for each \( u \in \pre{<\kappa}{\kappa} \) a triple
\[
( s_u,t_u,\gamma_u)\in{}^{{<}\kappa}\kappa\times T\times\kappa
\]
such that the following statements hold for all $u,v\in {}^{{<}\kappa}\kappa$:
\begin{enumerate-(a)} 
 \item if $u\subsetneq v$, then $s_u\subsetneq s_v$ and $t_u\subsetneq t_v$;

 \item if $u$ and $v$ are incompatible in ${}^{{<}\kappa}\kappa$, then $s_u$ and $s_v$ are incompatible in ${}^{{<}\kappa}\kappa$ and $t_u$ and $t_v$ are incompatible in $T$; 
 \item if $\length{s}\in\Lim$, then $s_u=\bigcup\Set{s_{u\restriction\alpha}}{\alpha<\length{u}}$ and $t_u=\bigcup\Set{t_{u\restriction\alpha}}{\alpha<\length{u}}$; 

 \item \label{item:iv} $f[N_{s_u}]=N_{t_u} \cap[T]$;  
 
  \item \label{item:v} for all $\alpha,\beta<\kappa$,  $\gamma_u<\length{t_{u^\smallfrown \alpha }}, \length{t_{u^\smallfrown \beta }}$,
   $t_{u^\smallfrown \alpha }\restriction\gamma_u=t_{u^\smallfrown \beta }\restriction\gamma_u$, and $t_{u^\smallfrown \alpha }(\gamma_u)\neq t_{u^\smallfrown \beta }(\gamma_u)$.  
\end{enumerate-(a)} 
Condition~\ref{item:iv} ensures that the construction can be carried over limit levels (still having \( t_u \in T \)), while condition~\ref{item:v} gives us cofinally many \( \kappa \)-splitting nodes. Thus  
\[
S  =   \Set{t\in T}{\exists u\in{}^{{<}\kappa}\kappa \, (t \subseteq t_u)}
\]
 is a $\kappa$-Miller subtree of $T$.  
\end{proof} 

An important corollary of Proposition~\ref{proposition:MillerTreeHomeoKappaKappa} is the following result.

\begin{theorem}\label{theorem:MillTreeHD}
 Let $\kappa$ be a weakly compact cardinal and $A$ be a subset of ${}^\kappa\kappa$. Then $A$ satisfies the Hurewicz dichotomy if and only if it satisfies the Miller-tree Hurewicz dichotomy. 
\end{theorem}

Moreover, using Corollary~\ref{corollary:RevNegroPonte} and Lemma~\ref{lemma:PerfectClosedTreeEmb} in the non-weakly compact case and 
Proposition~\ref{proposition:MillerTreeHomeoKappaKappa} 
in the weakly compact case, we can show that the Hurewicz dichotomy is always equivalent to an apparently stronger statement. 

\begin{corollary} 
Let $A$ be a subset of ${}^\kappa\kappa$. Then $A$ satisfies the Hurewicz dichotomy if and only if either $A$ is contained in a $\K\kappa$ subset of ${}^\kappa\kappa$ or there is a \emph{superclosed} subtree $T \subseteq {}^{{<}\kappa}\kappa$ such that $[T]\subseteq A$ and $[T]$ is homeomorphic to ${}^\kappa\kappa$. 
\end{corollary}


\section{The Hurewicz dichotomy at a given cardinal $\kappa$} \label{sec:HDatagivenkappa}

This section is mainly devoted to the proof of the consistency of the Hurewicz dichotomy for \( \mathbf{\Sigma}^1_1 \) subsetes of  \( \pre{\kappa}{\kappa} \) at a given uncountable \( \kappa \) satisfying \( \kappa = \kappa^{< \kappa} \) (Theorem~\ref{theorem:Cons General}), and of its strengthenings (Theorems~\ref{theorem:ConsNonWC} and ~\ref{thm:weaklycompactintroduction}).

\subsection{The non-weakly compact case}
 \label{sec:Hurewicznonweaklycompact}

We first consider the special case where $\kappa$ is not weakly compact.

\begin{definition} 
Given a subset  $A$ of ${}^\kappa\kappa$, 
we let $\KK(A)$ denote the partial order defined by the following clauses. 
\begin{enumerate-(i)}
 \item Conditions in $\KK(A)$ are pairs $p=\langle\alpha_p,c_p\rangle$ such that $\alpha_p<\kappa$ 
   and $\pmap{c_p}{A}{\kappa}{part}$ is a partial function of cardinality smaller than $ \kappa$. 
   Given  $p \in \KK(A)$ and $\gamma\in\ran{c_p}$, we define 
   \[
T_p(\gamma)  =  \Set{x\restriction\beta}{\beta\leq\alpha_p,  x\in\dom{c_p},  c_p(x)=\gamma}.
\] 

 \item Given $p, q \in \KK(A)$, we set $p\leq_{\KK(A)}q$ if and only if  $\alpha_q\leq\alpha_p$, $c_q\subseteq c_p$ and $T_p(\gamma)$ is an end-extension of $T_q(\gamma)$ for every $\gamma\in\ran{c_q}$, that is \( {T_p(\gamma) \cap \pre{\alpha_q+1}{\kappa}} = T_q(\gamma) \). 
\end{enumerate-(i)}
\end{definition}

\begin{proposition}\label{proposition:LongCondInKA}
 Given $\alpha<\kappa$, the set 
\[
D_\alpha  =  \Set{p\in\KK(A)}{\alpha\leq\alpha_p} 
\]
is dense in $\KK(A)$.  
\end{proposition}

\begin{proof}
 If $p$ is a condition in $\KK(A)$ with $\alpha_p<\alpha$, then $q = \langle \alpha,c_p\rangle \in \KK(A)$ and $q \leq_{\KK(A)} p$. 
\end{proof}

Remember that a subset $D$ of a partial order $\PP$ is \emph{directed} if two conditions in $D$ have a common extension that is an element of $D$.

\begin{definition}
Let \( \PP \) be a partial order.  
\begin{enumerate-(i)}
\item We say that $\PP$ is 
\emph{${<}\kappa$-directed closed} if for every directed subset $D$ of $\PP$ of cardinality  ${<}\kappa$ 
there is a condition in $\PP$ that lies below all elements of $D$;
\item We say that $\PP$ is 
\emph{strongly $\kappa$-linked} if there is a function $\map{g}{\PP}{\kappa}$ such that \( p \) and \( q \) have a greatest lower bound $\mathrm{glb}_\PP(p,q)$ with respect to \( \PP \)
 for all $p,q\in\PP$ with $g(p)=g(q)$.  
\end{enumerate-(i)}
\end{definition}

\begin{lemma}
 The partial order $\KK(A)$ is ${<}\kappa$-directed closed and strongly $\kappa$-linked.  
\end{lemma}

\begin{proof}
 Let $D$ be a directed subset of $\KK(A)$ of cardinality less than $\kappa$. Define 
\( p = \langle \alpha_p, c_p \rangle \) by setting
\[
\alpha_p = \sup\Set{\alpha_q}{q\in D} \quad \text{and} \quad c_p =  \bigcup\Set{c_q}{q\in D}.
\]
 Then $p$ is a condition in $\KK(A)$ (the map \( c_p \) is well-defined because \( D \) is directed). 
We claim that \( p \leq_{\KK(A)} q \) for all \( q \in D \). By definition of \( p \), given such a \( q \) we need only to show that \( T_p(\gamma) \) is an end-extension of \( T_q(\gamma) \) for all \( \gamma \in \ran{c_q}  \subseteq  \ran{c_p}$. Fix  $x\in\dom{c_p}$ with $c_p(x)=\gamma$. By definition of \( c_p \), there is \( q' \in D \) with \( x \in \dom{c_{q'}} \), and since \( D \) is directed there is \( r \in D \) with \( r \leq_{\KK(A)} q,q' \), so that 
 $x\in\dom{c_r}$. Since  \( c_r \subseteq c_p \), we have \( x \restriction \beta \in T_r(\gamma) \) for all \( \beta \leq \alpha_q \leq \alpha_r \), and since
$T_r(\gamma)$ is an end-extension of $T_q(\gamma)$ we get \( x \restriction \beta \in T_q(\gamma) \) for all \( \beta \leq \alpha_q \), as required.
 
 Now we prove that \( \KK(A) \) is strongly \( \kappa \)-linked. Fix a function $\map{f}{\KK(A)}{\kappa}$ such that $f(p)=f(q)$ if and only if $\alpha_p=\alpha_q$, $\ran{c_p}=\ran{c_q}$, and $T_p(\gamma)=T_q(\gamma)$ for every $\gamma\in\ran{c_p}$ (such an \( f \) exists because we are assuming that  $\kappa=\kappa^{{<}\kappa}$). Set
\[
D  =  \Set{p\in\KK(A)}{\forall x,y\in\dom{c_p} \, (x\neq y \Rightarrow  x\restriction\alpha_p\neq y\restriction\alpha_p)}.
\]
Using the fact that \( c_p \) has cardinality \( {<} \kappa \) and arguing as in the proof of Proposition~\ref{proposition:LongCondInKA}, $D$ is a dense subset of $\KK(A)$. Fix any map \( \KK(A) \to D;\ p \mapsto \bar{p} \) such that \( \bar{p} \leq_{\KK(A)} p \), and let \( \goedel{\cdot}{\cdot} \) be the usual G\"odel pairing function. Finally, set 
\[ 
\Map{g}{\KK(A)}{\kappa}{p}{\goedel{f(p)}{f(\bar{p})}}.
 \] 
We claim that \( g \) witnesses that \( \KK(A) \) is \( \kappa \)-linked.

First we show that if \( f(\bar{p}) = f(\bar{q}) \), then \( p \) and \( q \) are compatible in \( \KK(A) \). Fix $x\in\dom{c_{\bar{p}}}\cap\dom{c_{\bar{q}}}$, and set \( \gamma = c_{\bar{p}}(x) \). 
 Since $x\restriction\alpha_{\bar{p}}\in T_{\bar{p}}(\gamma)=T_{\bar{q}}(\gamma)$, there is  $y\in\dom{c_{\bar{q}}}$ with $c_{\bar{q}}(y) = \gamma$ and $x\restriction\alpha_{\bar{p}}=y\restriction\alpha_{\bar{p}}$. 
 By the definition of $D$, we can conclude that  $x=y$, and hence $c_{\bar{q}}(x) = \gamma = c_{\bar{p}}(x)$. This shows that $r=\langle\alpha_{\bar{p}},c_{\bar{p}}\cup c_{\bar{q}}\rangle$ is a condition in $\KK(A)$ that extends both $\bar{p}$ and $\bar{q}$, and hence it also extends both \( p \) and \( q \). 

Finally we show that if $p,q \in \KK(A)$ are compatible and such that $f(p)=f(q)$, then \( \mathrm{glb}_{\KK(A)}(p,q) \) exists. Set $r=\langle\alpha_p, c_p\cup c_q\rangle$. Then $r \in \KK(A)$,
 $\ran{c_p}=\ran{c_q}=\ran{c_r}$, and $T_p(\gamma)=T_q(\gamma)=T_r(\gamma)$ for all $\gamma\in\ran{c_r}$. 
 This shows that $r\leq_{\KK(A)}p,q$ and  $s\leq_{\KK(A)} r$ for all $s\in\KK(A)$ with $s\leq_{\KK(A)}p,q$: thus 
 $r=\mathrm{glb}_{\KK(A)}(p,q)$.  
\end{proof}

\begin{proposition}\label{proposition:CodeForceDenseElements}
 Given $x\in A$, the set 
\[
D_x  =  \Set{p\in\KK(A)}{x\in\dom{c_p}}
\]
 is dense in $\KK(A)$. 
\end{proposition}

\begin{proof}
 Given a condition $p$ in $\KK(A)$ with $x\notin\dom{c_p}$, let $\gamma<\kappa$ be such that $\gamma\notin\ran{c_p}$. Define \( \alpha_q = \alpha_p \) and \( \map{c_q} {\dom{c_p} \cup \{ x \}}{\kappa} \) by setting \( c_q(x) = \gamma \) and \( c_q(y) = c_p(y) \) for all \( y \in \dom{c_p} \). 
 Then $q = \langle \alpha_q,c_q \rangle \in \KK(A)$ and $q \leq_{\KK(A)} p$. 
\end{proof}

\begin{definition}
If $G$ is $\KK(A)$-generic over the ground model $\VV$, then we set 
\[
\map{c_G  =  \bigcup\Set{c_p}{p\in G}}{A}{\kappa}.
\]
\end{definition}

By Proposition~\ref{proposition:CodeForceDenseElements}, \( c_G \) is a total function, and using a similar density argument it can be shown that it is also surjective.

The proof of the following lemma uses ideas contained in {\cite[Section 5]{MR631563}}.

\begin{lemma}\label{lemma:TreesAreKurepaNoAronszajn}
 Let $\dot{\PP}$ be a $\KK(A)$-name for a ${<}\kappa$-closed partial order and let $G*H$ be $(\KK(A)*\dot{\PP})$-generic over the ground model $\VV$. 
 If  $\gamma<\kappa$, then 
\[
T_G(\gamma)  =  \Set{x\restriction\beta}{ \beta<\kappa, \, x\in A, \,  c_G(x)=\gamma}
\]
 is a pruned $\kappa$-tree without $\kappa$-Aronszajn subtrees in $\VV[G,H]$. 
\end{lemma}

\begin{proof}
 Since $T_G(\gamma)=\bigcup\Set{T_p(\gamma)}{p\in G, \, \gamma\in \ran{c_p}}$, Proposition~\ref{proposition:LongCondInKA} 
 and the definition of $\KK(A)$ imply that each $T_G(\gamma)$ is a pruned $\kappa$-tree in \( \VV[G,H] \).

 Let $\dot{T}_\gamma$ denote the canonical $(\KK(A)*\dot{\PP})$-name for $T_G(\gamma)$ 
 and assume, towards a contradiction, that there is a $(\KK(A)*\dot{\PP})$-name $\dot{S}$ 
 and a condition $\langle p_0,\dot{q}_0\rangle$ in $G*H$ that forces $\dot{S}$ to be a $\kappa$-Aronszajn subtree of $\dot{T}_\gamma$. 
 Since $\kappa$ remains regular in every $(\KK(A)*\dot{\PP})$-generic extension, 
 we can find a descending sequence $\seq{\langle p_n,\dot{q}_n\rangle}{n<\omega}$ of conditions in $\KK(A)*\dot{\PP}$  such that $\alpha_{p_n}<\alpha_{p_{n+1}}$ and 
 \[
\langle p_{n+1},\dot{q}_{n+1}\rangle\Vdash \anf{\check{x}\restriction\check{\alpha}_{p_{n}}\notin\dot{S}}
\]  
 for all $n<\omega$ and $x\in\dom{c_{p_n}}$. Define  
 \[
p  =  \left\langle \sup\Set{\alpha_{p_n}}{n<\omega}, \, \bigcup\Set{c_{p_n}}{n<\omega}\right \rangle.
\]
 Then $p$ is a condition in $\KK(A)$ that is stronger than $p_n$ for every $n<\omega$. 
 This allows us to find  a $\KK(A)$-name $\dot{q}$ for a condition in $\dot{\PP}$ such that $\langle p,\dot{q}\rangle\leq_{\KK(A)*\dot{\PP}}\langle p_n,\dot{q}_n\rangle$ holds for all $n<\omega$.  
 Our construction ensures that $\langle p,\dot{q}\rangle\Vdash\anf{\dot{T}_\gamma\cap{}^{\check{\alpha}_p}\check{\kappa}=\check{R}}$, where 
 \[
R  =  \Set{x\restriction\alpha_p}{x\in\dom{c_p}, \, c_p(x)=\gamma}.
\]
But we also ensured that $\langle p,\dot{q}\rangle\Vdash\anf{\check{s}\notin\dot{S}}$ holds for every $s\in R$. 
Since $\langle p,\dot{q}\rangle \leq_{\KK(A) * \dot{\PP}} \langle p_0, \dot{q}_0 \rangle$ also forces that $\dot{S}$ intersects every level of $\dot{T}_\gamma$, this yields a contradiction. 
\end{proof}

\begin{lemma}\label{lemma:CodedSetIsKKappa}
  Let $\dot{\PP}$ is a $\KK(A)$-name for a ${<}\kappa$-closed partial order and $G*H$ be $(\KK(A)*\dot{\PP})$-generic over the ground model $\VV$. Then 
 \[
A  =  \bigcup\Set{[T_G(\gamma)]^{\VV[G,H]}}{\gamma<\kappa}.
\]
  In particular, $A$ is a $\K\kappa$ subsets of ${}^\kappa\kappa$ in $\VV[G,H]$. 
\end{lemma} 

\begin{proof}
 By Proposition~\ref{proposition:CodeForceDenseElements}, we have   $A\subseteq \bigcup\Set{[T_G(\gamma)]^{\VV[G,H]}}{\gamma<\kappa}$. 
 Assume, towards a contradiction, that there is a $\gamma<\kappa$ and an $x\in[T_G(\gamma)]^{\VV[G,H]}$ with $x\notin A$. 
 Pick a $(\KK(A)*\dot{\PP})$-name $\dot{x}$ for an element of ${}^\kappa\kappa$ with $x=\dot{x}^{G*H}$ 
 and let $\dot{T}_\gamma$ be the canonical $(\KK(A)*\dot{\PP})$-name for $T_G(\gamma)$. 
 Then we have $\langle p_0,\dot{q}_0\rangle\Vdash\anf{\dot{x}\in[\dot{T}_\gamma]\setminus\check{A}}$ for some condition $\langle p_0,\dot{q}_0\rangle$ in $G*H$. 
 Since $\kappa$ remains regular in every $(\KK(A)*\dot{\PP})$-generic extension, we can find a descending sequence $\seq{\langle p_n,\dot{q}_n\rangle}{n<\omega}$ 
 of conditions in $\KK(A)*\dot{\PP}$ and a sequence $\seq{s_n\in{}^{{<}\kappa}\kappa}{n<\omega}$ 
 such that $\alpha_{p_n}<\length{s_n}<\alpha_{p_{n+1}}$, $\langle p_{n+1},\dot{q}_{n+1}\rangle\Vdash\anf{\check{s}_n\subseteq\dot{x}}$ 
 and $s_n\not\subseteq y$ for all $n<\omega$ and $y\in\dom{c_{p_n}}$.  
Set  
\[
p  =  \left\langle \sup\Set{\alpha_n}{n<\omega},  \bigcup\Set{c_{p_n}}{n<\omega}\right\rangle.
\]
  Then $p \in \KK(A)$ and $p \leq_{\KK(A)} p_n$ for every $n<\omega$,
  and hence we can find a $\KK(A)$-name $\dot{q}$ for a condition in $\dot{\PP}$ such that 
  $\langle p,\dot{q}\rangle\leq_{\KK(A)*\dot{\PP}}\langle p_n,\dot{q}_n\rangle$ holds for all $n<\omega$.  
  Set $s=\bigcup\Set{s_n}{n<\omega}\in{}^{\alpha_p}\kappa$. Then $s\notin T_p(\gamma)$ and $\langle p,\dot{q}\rangle\Vdash\anf{\check{s}\subseteq\dot{x}}$. 
  Hence $\langle p,q\rangle\Vdash\anf{\dot{x}\notin\dot{T}_\gamma}$, a contradiction. 
\end{proof}

Let $T \subseteq {}^{{<}\kappa}\kappa\times{}^{{<}\kappa}\kappa$ be a subtree and $A=p[T]$ (where both the body of \( T \) and its projection are computed in the ground model \( \VV \)).  
We will now show that in every $\KK(A)$-generic extension $\VV[G]$ of $\VV$, 
either $p[T]^{\VV[G]}$ is a $\K{\kappa}$ set or $p[T]^{\VV[G]}$ contains a perfect subset.

\begin{definition}
 Given a subtree  $T \subseteq {}^{{<}\kappa}\kappa\times{}^{{<}\kappa}\kappa$, we call a function  
 \[
\map{\iota}{{}^{<\kappa}2}{T} \colon s \mapsto \langle t^s_0, t^s_1 \rangle
\]
 a \emph{$\exists^x$-perfect embedding into $T$} if the following statements hold for all $s,s'\in{}^{{<}\kappa}2$ and \( i = 0,1 \).
 \begin{enumerate-(i)}
  \item If $s\subsetneq s'$, then $t^s_i\subsetneq t^{s'}_i$.
 \item If $s$ and $s'$ are incompatible in ${}^{<\kappa}2$, then $t^s_0$ and $t^{s'}_0$ are incompatible in ${}^{<\kappa}\kappa$.
 \item If $\length{s}\in\Lim$, then 
  \begin{equation*}
   t^s_i=\bigcup\Set{t^{s \restriction \alpha}_i}{\alpha<\length{s} }.
  \end{equation*} 
 \end{enumerate-(i)}
\end{definition}

\begin{proposition}\label{proposition:perfect embedding} 
 Let $T$ be a subtree of ${}^{{<}\kappa}\kappa\times{}^{{<}\kappa}\kappa$. If there is a $\exists^x$-perfect embedding into $T$, then $p[T]$ contains a closed subset of ${}^\kappa\kappa$ homeomorphic to ${}^\kappa 2$. 
\end{proposition}

\begin{proof} 
We use the notation introduced in the previous definition. 
 Define $S=\Set{t\in{}^{{<}\kappa}\kappa}{\exists s\in {}^{<\kappa}2 \, ( t\subseteq t^s_0)}$. Then $S$ is a subtree of ${}^{{<}\kappa}\kappa$, and our assumption imply that the function 
 \[
\Map{i}{{}^\kappa 2}{[S]}{x}{\bigcup\Set{t^{x\restriction\alpha}_0}{\alpha<\kappa}},
\]
 is a homeomorphism between ${}^\kappa 2$ and the closed set $[S]$.  
 Given $x\in{}^\kappa 2$, the function $\bigcup\Set{t^{x\restriction\alpha}_1}{\alpha<\kappa}$ witnesses that $f(x)$ is an element $p[T]$.   
 This shows that $[S]$ is a closed-in-\( \pre{\kappa}{\kappa} \) subset of $p[T]$ homeomorphic to ${}^\kappa 2$. 
\end{proof}

\begin{lemma}[{\cite[Lemma 7.6]{PL}}] \label{lemma: characterization exists-perfect embedding} 
 The following statements are equivalent for every subtree $T$ of ${}^{{<}\kappa}\kappa\times{}^{{<}\kappa}\kappa$.
\begin{enumerate-(1)}
 \item There is a $\exists^x$-perfect embedding into $T$. 
 \item If $\PP$ is ${<}\kappa$-closed partial order such that forcing with $\PP$ adds a new subset of $\kappa$, then forcing with $\PP$ adds a new element of $p[T]$. 
\end{enumerate-(1)}
\end{lemma}


Before we present the proof of Theorem~\ref{theorem:ConsNonWC}, we need to show that a ${<}\kappa$-support iteration of partial orders of the form $\KK(A)$ preserves cardinals greater than $\kappa$. 
This is a consequence of the following lemma, whose proof is a small modification of an argument due to Baumgartner contained in the proof of  {\cite[Theorem 4.2]{MR823775}}.

\begin{lemma}\label{lemma:IterationCouplingRepresentable}
 Let \(\lambda\) be an ordinal, and let $\langle\seq{\vec{\PP}_{{<}\alpha}}{\alpha\leq\lambda},\seq{\dot{\PP}_\alpha}{\alpha<\lambda}\rangle$ be a forcing iteration with ${<}\kappa$-support such that 
 \begin{equation*}
   \mathbbm{1}_{\vec{\PP}_{{<}\alpha}}\Vdash\anf{\text{$\dot{\PP}_\alpha$ is a ${<}\check{\kappa}$-closed and strongly $\check{\kappa}$-linked}}
 \end{equation*}
 holds for all $\alpha<\lambda$. Then the partial order $\vec{\PP}_{{<}\lambda}$ is $\kappa^+$-Knaster.
\end{lemma}

\begin{proof}[Proof of Theorem~\ref{theorem:ConsNonWC}] 
Assume that \( \kappa \) is not weakly compact and set $\lambda=2^\kappa$.  
Since $\kappa$ remains an uncountable cardinal with $\kappa=\kappa^{{<}\kappa}$ after forcing with a ${<}\kappa$-support iteration of ${<}\kappa$-closed partial orders 
(see {\cite[Proposition 7.9]{MR2768691}}),  there is a forcing iteration  
\[
\vec{\PP}  =  \langle\seq{\vec{\PP}_{{<}\alpha}}{\alpha\leq\lambda},\seq{\dot{\PP}_\alpha}{\alpha<\lambda}\rangle
\]
with ${<}\kappa$-support and a sequence $\seq{\dot{b}_\alpha}{\alpha<\lambda}$ such that the following statements hold for all $\alpha<\lambda$:
\begin{enumerate-(a)}

 \item $\dot{b}_\alpha$ is a $\dot{\PP}_{{<}\alpha}$-name with the property that $\dot{b}_\alpha^G$ is a surjection from $(2^\kappa)^{\VV[G]}$ onto the set of all subtrees of 
  ${}^{{<}\kappa}\kappa\times{}^{{<}\kappa}\kappa$ in $\VV[G]$ whenever  $G$ is $\vec{\PP}_{{<}\alpha}$-generic over $\VV$;

 \item If $\alpha=\goedel{\beta}{\gamma}$, $G$ is $\vec{\PP}_{{<}\alpha}$-generic over $\VV$, $\bar{G}$ is the filter on $\vec{\PP}_{{<}\beta}$ induced by $G$, $T=\dot{b}_\beta^{\bar{G}}(\gamma)$, 
  and $A=p[T]^{\VV[G]}$, then $\dot{\PP}_\alpha^G=\KK(A)^{\VV[G]}$. 
\end{enumerate-(a)}

\begin{claim} \label{claim:vecPP}
 If $\alpha\leq\lambda$, then $\vec{\PP}_{{<}\alpha}$ is ${<}\kappa$-directed closed and $\kappa^+$-Knaster. 
\end{claim}

\begin{proof}[Proof of the Claim]
 This follows directly from  Lemma~\ref{lemma:IterationCouplingRepresentable} and {\cite[Proposition 7.9]{MR2768691}}. 
\end{proof}

\begin{claim}
 If $\alpha\leq\lambda$, then $\mathbbm{1}_{\vec{\PP}_{{<}\alpha}}\Vdash\anf{\check{\lambda}=2^{\check{\kappa}}}$. 
\end{claim}

\begin{proof}[Proof of the Claim]
 By the definition of $\KK(A)$, we know that $\dot{\PP}_\alpha$ is forced to be a subset of $\HHH{\kappa^+}$ in every $\vec{\PP}_{{<}\alpha}$-generic extension of $\VV$  for every $\alpha<\lambda$. 
 Since $\vec{\PP}_{{<}\alpha}$ satisfies the $\kappa^+$-chain condition by Claim~\ref{claim:vecPP}, this observation allows us to inductively prove that $\vec{\PP}_{{<}\alpha}$ has a dense subset of cardinality at most $\lambda$ 
 and $\mathbbm{1}_{\vec{\PP}_{{<}\alpha}}\Vdash\anf{\check{\lambda}=2^{\check{\kappa}}}$ holds for all $\alpha<\lambda$. 
\end{proof}

Let $G$ be $\vec{\PP}_{{<}\lambda}$-generic over $\VV$. Given $\alpha<\lambda$, we let $G_\alpha$ denote the filter on $\vec{\PP}_{{<}\alpha}$ induced by $G$.  
If $\kappa$ is not inaccessible in $\VV$, then $\kappa$ is not inaccessible in $\VV[G]$, because both models contain the same ${<}\kappa$-sequences. 
Moreover, if there is a $\kappa$-Aronszajn tree in $\VV$, then this tree remains a $\kappa$-Aronszajn tree in $\VV[G]$, because ${<}\kappa$-closed forcings add no branches to such trees (see, for example, \cite[Proposition 7.3]{PL}). 
We can conclude that $\kappa$ is not a weakly compact cardinal in $\VV[G]$ and, by Corollary~\ref{corollary:RevNegroPonte}, this shows that the spaces ${}^\kappa\kappa$ and ${}^\kappa 2$ are homeomorphic in $\VV[G]$.  
Pick a subtree $T \subseteq {}^{{<}\kappa}\kappa\times{}^{{<}\kappa}\kappa$ in $\VV[G]$ and define $A=p[T]^{\VV[G]}$.

First assume that in $\VV[G]$ there is an $\exists^x$-perfect embedding into $T$. Since $\kappa$ is not weakly compact in $\VV[G]$, 
the above remarks and Proposition~\ref{proposition:perfect embedding} imply that $A$ contains a closed subset homeomorphic to ${}^\kappa\kappa$ in $\VV[G]$. 
Thus $A$ satisfies the strong Hurewicz dichotomy in $\VV[G]$.

Now assume that in $\VV[G]$ there is no $\exists^x$-perfect embedding into $T$. 
Since $\vec{\PP}_{{<}\lambda}$ satisfies the $\kappa^+$-chain condition, we can find $\beta<\lambda$ with $T\in\VV[G_\beta]$. 
Since the above claim shows that $\lambda=(2^\kappa)^{\VV[G_\beta]}$, there is a $\gamma<\lambda$ with $T=\dot{b}_\beta^{G_\beta}(\gamma)$.  Define $\alpha=\goedel{\beta}{\gamma}$. 
Since {\cite[Proposition 7.12]{MR2768691}} shows that $\VV[G]$ is a forcing extension of $\VV[G_\alpha]$ by a ${<}\kappa$-closed forcing  
and our assumption implies that there are no $\exists^x$-perfect embedding into $T$ in $\VV[G_\alpha]$, 
by Lemma~\ref{lemma: characterization exists-perfect embedding} we get $A=p[T]^{\VV[G_\alpha]}$. 
Moreover, we have $\dot{\PP}_\alpha^{G_\alpha}=\KK(A)^{\VV[G_\alpha]}$, and that $\VV[G]$ is a forcing extension of $\VV[G_{\alpha+1}]$ by a ${<}\kappa$-closed forcing. 
Thus $A$ is a $\K{\kappa}$ subset of ${}^\kappa\kappa$ in $\VV[G]$ by Lemma~\ref{lemma:CodedSetIsKKappa}, and hence  $A$ satisfies the strong Hurewicz dichotomy in $\VV[G]$ also in this case. 
\end{proof}


\subsection{The general case}
\label{sec:Hurewiczweaklycompact}

In this section we develop analogues of the above construction for the weakly compact case in order to prove Theorem~\ref{theorem:Cons General}. 
The key component is the following generalization of \emph{Hechler forcing} to uncountable regular cardinals.

\begin{definition}[$\kappa$-Hechler forcing]
 Let $\HH(\kappa)$ denote the partial order defined by the following clauses. 
 \begin{enumerate-(i)}
  \item A condition in $\HH(\kappa)$ is a pair $p=\langle t_p,a_p\rangle$ such that $t_p\in{}^{\alpha_p}\kappa$ for some $\alpha_p<\kappa$ and $a_p\in[{}^\kappa\kappa]^{{<}\kappa}$. 

  \item Given  $p,q \in \HH(\kappa)$, we set $p\leq_{\HH(\kappa)}q$ if and only if $t_q\subseteq t_p$, $a_q\subseteq a_p$, and $x(\beta)<t_p(\beta)$ for all $x\in a_q$ and $\beta\in\alpha_p\setminus\alpha_q$. 
 \end{enumerate-(i)}
\end{definition}

In the following we list the relevant basic properties of $\HH(\kappa)$.

\begin{lemma}
 The partial order $\HH(\kappa)$ is ${<}\kappa$-directed closed and strongly $\kappa$-linked.  
\end{lemma}

\begin{proposition}
 Given $\alpha<\kappa$ and $x\in{}^\kappa\kappa$, the set 
 \[
D  =  \Set{p\in\HH(\kappa)}{\alpha<\alpha_p, \, x\in a_p}
\]
 is dense in $\HH(\kappa)$.   
\end{proposition}

\begin{definition}
We denote by $\dot{h}$ the canonical $\HH(\kappa)$-name such that 
\[
\dot{h}^G  =  \map{\bigcup\Set{t_p}{p\in G}}{\kappa}{\kappa}
\] 
whenever $G$ is $\HH(\kappa)$-generic over the ground model $\VV$.  
\end{definition}

\begin{proposition}
 If $G$ is $\HH(\kappa)$-generic over $\VV$, then $\dot{h}^G$ witnesses that $({}^\kappa\kappa)^\VV$ is an eventually bounded subset of ${}^\kappa\kappa$ in $\VV[G]$.  
\end{proposition}

Our next aim is to generalize Lemma~\ref{lemma: characterization exists-perfect embedding} to the weakly compact case.

\begin{definition}\label{definition:ExistsSuperperfectEmb}
 Let $T$ be a subtree of ${}^{{<}\kappa}\kappa\times{}^{{<}\kappa}\kappa$. A \emph{$\exists^x$-superperfect embedding into $T$} is a function 
 \[
\Map{\iota}{{}^{<\kappa}\kappa}{T}{s}{\langle t^s_0, t^s_1 \rangle}
\]
 with the property that the following statements hold for all $s,s^\prime\in{}^{{<}\kappa}2$ and \( i = 0,1 \).
 \begin{enumerate-(i)}
  \item If $s\subsetneq s^\prime$, then $t^s_i\subsetneq t^{s^\prime}_i$. 

  \item There is  $\length{t_0^s}\leq\gamma_s<\kappa$ such that  for all $\alpha<\beta<\kappa$
\begin{itemizenew}
\item
 $\gamma_s<\length{t_0^{s^\smallfrown  \alpha }}, \length{t_0^{s^\smallfrown  \beta }}$, \\
\item
   $t_0^{s^\smallfrown \alpha }\restriction\gamma_s=t_0^{s^\smallfrown \beta }\restriction\gamma_s$, and \\
\item
 $t_0^{s^\smallfrown \alpha }(\gamma_s)\neq t_0^{s^\smallfrown \beta }(\gamma_s)$. 
   \end{itemizenew}
  \item If $\length{s}\in\Lim$, then
  \begin{equation*}
   t^s_i=\bigcup\Set{t^{s \restriction \alpha}_i}{\alpha<\length{s} }.
  \end{equation*} 
 \end{enumerate-(i)}
\end{definition}

\begin{proposition}\label{proposition:superperfect embedding} 
 Let $T$ be a subtree of ${}^{{<}\kappa}\kappa\times{}^{{<}\kappa}\kappa$. If there is a $\exists^x$-superperfect embedding into $T$, then there is a $\kappa$-Miller tree $S$ such that $[S]\subseteq p[T]$. 
\end{proposition}

\begin{proof}
 We use the notation from the above definition. If we define 
 \[
S  =  \lrSet{t\in{}^{{<}\kappa}\kappa}{\exists s\in{}^{{<}\kappa}\kappa \,  (t \subsetneq  t_0^{s})},
\] 
 then $S$ is a $\kappa$-Miller tree. It remains to show that \( [S] \subseteq p[T] \). Pick $y\in[S]$. Then $x=\bigcup\Set{s\in{}^{{<}\kappa}\kappa}{t_0^s\subseteq y}$ is an element of ${}^\kappa\kappa$ with 
 $y=\bigcup\Set{t_0^{x\restriction\alpha}}{\alpha<\kappa}$ and $\left \langle y,  \bigcup\Set{t_1^{x\restriction\alpha}}{\alpha<\kappa}\right\rangle \in [T]$, so that \( y \in p[T] \). 
\end{proof}

\begin{corollary}\label{corollary:ProfPerfEmbImpliesNotKKappa}
 Let $T$ be a subtree of ${}^{{<}\kappa}\kappa\times{}^{{<}\kappa}\kappa$. If there is a $\exists^x$-superperfect embedding into $T$, then $p[T]$ is not contained in a $\K{\kappa}$ subset of ${}^\kappa\kappa$.  
\end{corollary}

\begin{proof}
 This follows from Propositions~\ref{proposition:MillerTreeHomeoKappaKappa} and~\ref{proposition:superperfect embedding}.
\end{proof}

The following result shows how $\exists^x$-superperfect embeddings can be constructed from names with certain splitting properties.

\begin{lemma}\label{lemma:Superperfect Embedding from Splitting names}
 Assume that $\kappa$ is an inaccessible cardinal. Let $T$ be a subtree of ${}^{{<}\kappa}\kappa\times{}^{{<}\kappa}\kappa$, $\PP$ be a ${<}\kappa$-closed partial order, and $\dot{x}$ be a $\PP$-name for an element of \( \pre{\kappa}{\kappa} \) such that $\mathbbm{1}_\PP\Vdash\anf{\dot{x} \in p[T]}$ and 
for every $q \in \PP$ and every $\beta<\kappa$ there is an ordinal $\beta\leq\gamma<\kappa$ and a sequence $\seq{q_\alpha}{\alpha<\kappa}$ of conditions in $\PP$ below $q$ 
 with $q_{\alpha}\Vdash\anf{\dot{x}(\check{\gamma})>\check{\alpha}}$ for all $\alpha<\kappa$. Then there is a $\exists^x$-superperfect embedding into $T$. 
\end{lemma}

\begin{proof}
 Let $\dot{y}$ be $\PP$-name for an element of ${}^\kappa\kappa$ with $\mathbbm{1}_\PP\Vdash\anf{\langle\dot{x},\dot{y}\rangle\in[\check{T}]}$. 
 We inductively define a $\exists^x$-superperfect embedding $\seq{\langle t_0^s,t_1^s\rangle}{s\in{}^{{<}\kappa}\kappa}$ into $T$ 
 and a sequence of conditions $\seq{p_s}{s\in{}^{{<}\kappa}\kappa}$ such that the following statements hold for all $s,s^\prime\in{}^{{<}\kappa}\kappa$. 
 \begin{enumerate-(a)} 
  \item If $s\subseteq s^\prime$, then $p_{s^\prime}\leq_\PP p_s$. 

  \item $p_s\Vdash\anf{\check{t}_0^s\subseteq\dot{x} \, \wedge \, \check{t}_1^s\subseteq\dot{y}}$. 
 \end{enumerate-(a)}

Assume that $s\in{}^{{<}\kappa}\kappa$ with $\length{s}\in\Lim$ and we already constructed $t_0^{s \restriction \alpha}$, $t_1^{s \restriction \alpha}$ and $p_{s\restriction\alpha}$ with the above properties for all $\alpha<\length{s}$. 
 Then we let $p_s$ denote a condition with $p\leq_\PP p_{s\restriction\alpha}$ for all $\alpha<\length{s}$ and define $t_i^s$ as in the last clause of Definition~\ref{definition:ExistsSuperperfectEmb}. 
 
 Now assume that we already constructed $t_0^s$, $t_1^s$ and $p_s$ with the above properties for some $s\in{}^{{<}\kappa}\kappa$.  
 By our assumptions, we can find an ordinal $\length{t_0^s}<\gamma<\kappa$ and sequences $\seq{t_i^\alpha\in{}^{\gamma+1}\kappa}{\alpha<\kappa, \, i= 0,1}$ and $\seq{p_\alpha\in \PP}{\alpha<\kappa}$ 
 such that $p_\alpha\leq_\PP p_s$, $p_\alpha\Vdash\anf{\check{t}_0^\alpha\subseteq\dot{x}\wedge\check{t}_1^\alpha\subseteq\dot{y}}$ and $t_\alpha(\gamma)\neq t_\beta(\gamma)$ for distinct $\alpha,\beta<\kappa$. 
 Since $\kappa$ is inaccessible, there is a $\gamma_s\leq\gamma$ and an injection $\map{f}{\kappa}{\kappa}$ with $t_0^{f(\alpha)}\restriction\gamma_s=t_0^{f(\beta)}\restriction\gamma_s$ 
 and $t_0^{f(\alpha)}\neq t_0^{f(\beta)}$ for distinct $\alpha,\beta<\kappa$. Define $t_i^{s^\smallfrown \alpha }=t_i^{f(\alpha)}$ and $p_{s^\smallfrown \alpha }=p_{f(\alpha)}$ for all $\alpha<\kappa$ and $i=0,1$.  
\end{proof}

\begin{lemma} \label{lemma:Hechler function and splitting name} 
 Let $p \in \HH(\kappa)$ 
 and $\dot{x}$ be a $\HH(\kappa)$-name for an element of ${}^\kappa\kappa$ with $p\Vdash\anf{\dot{x}\not\leq^* \dot{h}}$. If $q$ is a condition in $\HH(\kappa)$ below $p$ and $\beta<\kappa$, 
 then there is an ordinal $\beta\leq\gamma<\kappa$ and a sequence $\seq{q_\alpha}{\alpha<\kappa}$ of conditions in $\HH(\kappa)$ below $q$ such that $q_{\alpha}\Vdash\anf{\dot{x}(\check{\gamma})>\check{\alpha}}$ holds for all $\alpha<\kappa$. 
\end{lemma} 

\begin{proof} 
 Assume towards a contradiction that the above implication fails for some $q\leq_{\HH(\kappa)}p$ and $\beta<\kappa$. 
 Given an ordinal $\beta\leq\gamma<\kappa$, this assumption implies that there is an $\alpha_\gamma<\kappa$ such that $r\Vdash\anf{\dot{x}(\check{\gamma})<\check{\alpha}_\gamma}$ holds for every extension $r$ of $q$ in $\HH(\kappa)$ 
 that decides $\dot{x}(\gamma)$. This allows us to define a function $g\in{}^\kappa\kappa$ such that $q\Vdash\anf{\dot{x}\leq^* \check{g}}$. 
 If we define $r=\langle t_q,a_q\cup\{g\}\rangle$, then $r$ is a condition in $\HH(\kappa)$ with $r\leq_{\HH(\kappa)}q$ and $r\Vdash\anf{\dot{x}\leq^*\check{g}\leq^*\dot{h}}$, a contradiction.  
\end{proof}

The next lemma summarizes the above observations.

\begin{lemma}\label{lemma: equivalent conditions for superperfect embedding} 
 Let $p$ be a condition in $\HH(\kappa)$ with $p\Vdash\anf{\text{$\check{\kappa}$ is weakly compact}}$.\footnote{Since $\HH(\kappa)$ is ${<}\kappa$-closed, this assumption implies that $\kappa$ is weakly compact (and therefore inaccessible) in $\VV$.} 
 Then the following statements are equivalent for every subtree $T$ of ${}^{{<}\kappa}\kappa\times{}^{{<}\kappa}\kappa$.  
 \begin{enumerate-(1)}
  \item\label{condition: ProjPerfectEmbedding} There is a $\exists^x$-superperfect embedding into $T$. 
 
  \item\label{condition: not K-kappa in Hechler Extensions} $p[T]$ is not contained in a $\K{\kappa}$ subset of ${}^\kappa\kappa$ in any generic extension of the ground model $\VV$ by a ${<}\kappa$-closed forcing. 
 
  \item\label{condition: not K-kappa in generic extension} $p[T]$ is not contained in a $\K{\kappa}$ subset of ${}^\kappa\kappa$ in any $\mathbb{H}(\kappa)$-generic extension of the ground model $\VV$. 
 
  \item\label{condition: existence of name} $p\Vdash\anf{\exists x\in p[T] \, ( x\not\leq^* \dot{h})}$. 
  
  \item\label{condition: existence of Hechler name} There is an $\HH(\kappa)$-name $\dot{x}$ for an element of \( \pre{\kappa}{\kappa} \) such that  $\mathbbm{1}_{\HH(\kappa)} \Vdash\anf{\dot{x} \in p[T]}$ and for all $q\leq_{\HH(\kappa)}p$ and $\beta<\kappa$
   there is an ordinal $\beta\leq\gamma<\kappa$ and a sequence $\seq{q_\alpha}{\alpha<\kappa}$ of extensions of $q$ in $\HH(\kappa)$ such that $q_\alpha\Vdash\anf{\dot{x}(\check{\gamma})>\check{\alpha}}$ for all $\alpha<\kappa$.   

  \item\label{condition: existence of closed po name} There is a ${<}\kappa$-closed partial order $\PP$ and a $\PP$-name $\dot{x}$ for an element of \( \pre{\kappa}{\kappa} \) such that  $\mathbbm{1}_{\HH(\kappa)} \vDash\anf{\dot{x} \in p[T]}$ and for all $q\leq_\PP p$ and $\beta<\kappa$, 
   there is an ordinal $\beta\leq\gamma<\kappa$ and a sequence $\seq{q_\alpha}{\alpha<\kappa}$ of extensions of $q$ in $\PP$ such that $q_\alpha\Vdash\anf{\dot{x}(\check{\gamma})>\check{\alpha}}$ for all $\alpha<\kappa$.  
\end{enumerate-(1)}
\end{lemma}

\begin{proof} 
 The implications from~\ref{condition: not K-kappa in Hechler Extensions} to~\ref{condition: not K-kappa in generic extension} and from~\ref{condition: existence of Hechler name} to~\ref{condition: existence of closed po name} are trivial. 
 A combination of Lemma~\ref{lemma:Hechler function and splitting name} and the maximality principle shows that~\ref{condition: existence of name} implies~\ref{condition: existence of Hechler name}. 
 The implication from~\ref{condition: existence of closed po name} to~\ref{condition: ProjPerfectEmbedding} follows from Lemma~\ref{lemma:Superperfect Embedding from Splitting names} and our assumption. 
 Moreover, the implication from~\ref{condition: ProjPerfectEmbedding} to~\ref{condition: not K-kappa in Hechler Extensions} follows directly from Corollary~\ref{corollary:ProfPerfEmbImpliesNotKKappa} 
 and the observation that a $\exists^x$-superperfect embedding into $T$ remains $\exists^x$-superperfect after forcing with a ${<}\kappa$-distributive forcing.  
 Finally, assume that~\ref{condition: not K-kappa in generic extension} holds and let $G$ be $\HH(\kappa)$-generic over $\VV$ with $p\in G$. By our assumptions, $\kappa$ is weakly compact in $\VV[G]$ 
 and $p[T]$ is not contained in a $\K{\kappa}$ subset of ${}^\kappa\kappa$. In this situation, Lemma~\ref{lemma:bounded} shows that $p[T]$ is not eventually bounded in $\VV[G]$, and thus there is an $x\in p[T]$ with $x\not\leq^*\dot{h}^G$. 
 This shows that~\ref{condition: existence of name} holds in this case. 
\end{proof}

Note that if $\kappa$ is an inaccessible cardinal, then  the only implication in Lemma~\ref{lemma: equivalent conditions for superperfect embedding} 
requiring the assumption $p\Vdash\anf{\text{$\check{\kappa}$ is weakly compact}}$ is 
\ref{condition: not K-kappa in generic extension} \( \Rightarrow \)~\ref{condition: existence of name}. 

We are now ready to prove Theorem~\ref{theorem:Cons General} by combining all the previous results from this section.

\begin{proof}[Proof of Theorem~\ref{theorem:Cons General}]
Let
 \[
 \vec{\PP}  =  \langle\seq{\vec{\PP}_{{<}\alpha}}{\alpha\leq\kappa^+},\seq{\dot{\PP}_\alpha}{\alpha<\kappa^+}\rangle
\]
be the ${<}\kappa$-support forcing iteration such that the following statements hold for all $\alpha<\lambda$.
\begin{enumerate-(a)}
 \item If $\alpha$ is even, then $\mathbbm{1}_{\vec{\PP}_{{<}\alpha}}\Vdash\anf{\dot{\PP}_\alpha=\HH(\check{\kappa})}$.  

 \item If $\alpha$ is odd, then $\mathbbm{1}_{\vec{\PP}_{{<}\alpha}}\Vdash\anf{\dot{\PP}_\alpha=\KK({}^{\check{\kappa}}\check{\kappa})}$. 
\end{enumerate-(a)}
 We claim that there is an isomorphic copy \( \PP(\kappa) \) of $\PP_{{<}\kappa^+}$ that satisfies the statements listed in the theorem.  

 By Lemma~\ref{lemma:IterationCouplingRepresentable}, we know that for each $\alpha\leq\kappa^+$ the partial order $\vec{\PP}_{{<}\alpha}$ is ${<}\kappa$-directed closed and $\kappa^+$-Knaster, because it is a ${<}\kappa$-support iteration of ${<}\kappa$-directed closed and strongly $\kappa$-linked forcings. This shows that~\ref{theorem:Cons General-1} of Theorem~\ref{theorem:Cons General} is satisfied.

 In order to show that also~\ref{theorem:Cons General-2} of Theorem~\ref{theorem:Cons General} is satisfied, it suffices to show that each proper initial segment $\vec{\PP}_{{<}\alpha}$ of our iteration is a subset of $\HHH{\kappa^+}$ 
 and those initial segments are uniformly definable over $\langle\HHH{\kappa^+},\in\rangle$ in the parameter $\alpha$, because we can then replace $\vec{\PP}_{{<}\kappa^+}$ by 
 an isomorphic copy $\PP(\kappa)$ that is a definable subset of $\HHH{\kappa^+}$.  
 Let $\alpha<\kappa^+$ such that $\vec{\PP}_{{<}\bar{\alpha}}\subseteq\HHH{\kappa^+}$ for every $\bar{\alpha}<\alpha$. 
 If $\alpha\in\Lim$, then this assumption directly implies that $\vec{\PP}_{{<}\alpha}$ is a subset of $\HHH{\kappa^+}$. 
 Now assume that instead $\alpha=\bar{\alpha}+1$. If $\dot{q}$ is a $\vec{\PP}_{{<}\bar{\alpha}}$-name for a condition in $\vec{\PP}_{\bar{\alpha}}$, then $\dot{q}$ is forced to be an element of $\HHH{\kappa^+}$ 
 and hence we can find an equivalent canonical name (see \cite[Fact 3.6]{MR1234283}) that is an element of $\HHH{\kappa^+}$, because $\vec{\PP}_{{<}\bar{\alpha}}$ satisfies the $\kappa^+$-chain condition. 
 This shows that also in this case $\vec{\PP}_{{<}\alpha}\subseteq\HHH{\kappa^+}$. Next, note that the domains and the orderings of both $\KK({}^\kappa\kappa)$ and $\HH(\kappa)$ are definable over $\langle\HHH{\kappa^+},\in\rangle$ 
 by $\Sigma_1$-formulas with parameter $\kappa$. Hence, $\vec{\PP}_{{<}\alpha+1}$ is uniformly definable over $\langle\HHH{\kappa^+},\in,\vec{\PP}_{{<}\alpha}\rangle$. 
 This allows us to conclude that the initial segments of our iteration are uniformly definable over $\langle\HHH{\kappa^+},\in\rangle$. 

Finally, let us show that~\ref{theorem:Cons General-3} of Theorem~\ref{theorem:Cons General} is satisfied as well.
 Let $G$ be $\vec{\PP}_{{<}\kappa^+}$-generic over $\VV$. Given $\alpha<\kappa^+$, let $G_\alpha$ denote the filter on $\vec{\PP}_{{<}\alpha}$ induced by $G$, and $G^\alpha$ denote the filter in $\dot{\PP}_\alpha^{G_\alpha}$ induced by $G$. 
 Fix a subtree $T \subseteq {}^{{<}\kappa}\kappa\times{}^{{<}\kappa}\kappa$ in $\VV[G]$ such that in $\VV[G]$ the set $p[T]$ is not contained in a $\K{\kappa}$ subset of ${}^\kappa\kappa$ . 
 Since $\vec{\PP}_{{<}\kappa^+}$ satisfies the $\kappa^+$-chain condition in $\VV$, there is an $\alpha_0<\kappa^+$ with $T\in\VV[G_{\alpha_0}]$. We distinguish two cases, depending on whether \( \kappa \) is weakly compact in the forcing extension or not.

 First assume that $\kappa$ is weakly compact in $\VV[G]$.   
 By Lemma~\ref{lemma:bounded}, our assumptions imply that $p[T]$ is not eventually bounded in $\VV[G]$.  Pick an even $\alpha_0<\alpha<\kappa^+$. 
 Then we know that $\kappa$ is weakly compact in $\VV[G_\alpha,G^\alpha]$ and $p[T]$ is not eventually bounded in $\VV[G_\alpha,G^\alpha]$, 
 because $\VV[G]$ is a forcing extension of $\VV[G_\alpha,G^\alpha]$ by a ${<}\kappa$-closed forcing and both statements can be expressed by $\Pi_1$ (hence downward absolute) formulas over $\langle\HHH{\kappa^+},\in\rangle$.  
 This shows that there is a condition $p\in G^\alpha\subseteq\HH(\kappa)^{\VV[G_\alpha]}$ such that  
 \[
p \Vdash \anf{\text{$\kappa$ is weakly compact} \,  \wedge \,  \exists x\in p[\check{T}] \, (x\not\leq^*\dot{h})}
\]
 holds in $\VV[G_\alpha]$. By Lemma~\ref{lemma: equivalent conditions for superperfect embedding} there is a $\exists^x$-superperfect embedding into $T$ in $\VV[G_\alpha]$, and this embedding remains $\exists^x$-superperfect in $\VV[G]$. 
 By Proposition~\ref{proposition:superperfect embedding}  we can conclude that  in $\VV[G]$ the set $p[T]$ contains the body of a \( \kappa \)-Miller tree, which is a closed subset homeomorphic to ${}^\kappa\kappa$ by Proposition~\ref{proposition:MillerTreeHomeoKappaKappa}.

 Now assume that $\kappa$ is not weakly compact in $\VV[G]$. Pick an odd $\alpha_0<\alpha<\kappa^+$. 
 Since $\dot{\PP}_\alpha^{G_\alpha}=\KK({}^\kappa\kappa)^{\VV[G_\alpha]}$ and $\VV[G]$ is a forcing extension of $\VV[G_\alpha,G^\alpha]$ by a ${<}\kappa$-closed forcing, 
 Lemma~\ref{lemma:CodedSetIsKKappa} implies that $({}^\kappa\kappa)^{\VV[G_\alpha]}$ is a $\K{\kappa}$ subset of ${}^\kappa\kappa$ in $\VV[G]$, and therefore $p[T]^{\VV[G_\alpha]}\subsetneq p[T]^{\VV[G]}$. 
By Lemma~\ref{lemma: characterization exists-perfect embedding}, it follows that in $\VV[G_\alpha]$ there is a $\exists^x$-perfect embedding into $T$,  and this embedding remains $\exists^x$-perfect in $\VV[G]$. 
 Hence by Proposition~\ref{proposition:perfect embedding} the set  $p[T]$ contains a closed subset homeomorphic to ${}^\kappa 2$ in $\VV[G]$. 
 Since $\kappa$ is not weakly compact in $\VV[G]$, by Corollary~\ref{corollary:RevNegroPonte} such a set is also homeomorphic to ${}^\kappa\kappa$ in $\VV[G]$ and we are done. 
\end{proof}

\subsection{The Hurewicz dichotomy in Silver models}

A small variation of the above argument allows us to also show that every $\mathbf{\Sigma}^1_1$ subset of ${}^\kappa\kappa$ satisfies the 
Hurewicz dichotomy after forcing with $\mathrm{Col}(\kappa,\lambda)$ whenever $\lambda$ is an inaccessible cardinal greater than $\kappa$ (independently of whether \( \kappa \) is weakly compact in the forcing extension or not). 

\begin{theorem}\label{theorem:HDinCollapse}
 Let $\kappa$ be an uncountable regular cardinal, $\lambda>\kappa$ be an inaccessible cardinal and $G$ be $\mathrm{Col}(\kappa,\lambda)$-generic over $\VV$. 
 Then in $\VV[G]$ every $\mathbf{\Sigma}^1_1$ subset of ${}^\kappa\kappa$ satisfies the Hurewicz dichotomy.  
\end{theorem}

\begin{proof}
 Let 
$T \subseteq {}^{{<}\kappa}\kappa\times{}^{{<}\kappa}\kappa$ be a subtree in $\VV[G]$ such that $p[T]$ is not contained in a $\K{\kappa}$ subset of ${}^\kappa\kappa$ in $\VV[G]$. 
 Then there is a forcing extension $\WW$ of $\VV$ such that $T\in\WW$ and $\VV[G]=\WW[H]$ is a $\mathrm{Col}(\kappa,\lambda)$-generic extension of $\WW$. 
 First, assume that $\kappa$ is not weakly compact in $\VV[G]$. Since $({}^\kappa\kappa)^\WW$ has cardinality $\kappa$ in $\VV[G]$, 
 our assumption implies $p[T]^\WW\subsetneq p[T]^{\VV[G]}$, and 
 by Corollary~\ref{corollary:RevNegroPonte} and Proposition~\ref{proposition:perfect embedding}
 we can conclude that  $p[T]$ contains a closed subset homeomorphic to ${}^\kappa\kappa$ in $\VV[G]$. 

 Now assume that $\kappa$ is weakly compact in $\VV[G]$. By Lemma~\ref{lemma:bounded}, we know that $p[T]$ is not  eventually bounded in $\VV[G]$. 
 Then there are $H_0,H_1\in\VV[G]$ such that $H_0$ is $\HH(\kappa)$-generic over $\WW$, 
 $H_1$ is $\mathrm{Col}(\kappa,\lambda)$-generic over $\WW[H_0]$ and $\VV[G]=\WW[H_0,H_1]$. Since $\VV[G]$ is an extension of $\WW[H_0]$ by a ${<}\kappa$-closed forcing, 
 we know that $p[T]$ is not eventually bounded in $\WW[H_0]$. In particular, there is an $x\in p[T]^{\WW[H_0]}$ with $x\not\leq^*\dot{h}^{H_0}$. 
 By Proposition~\ref{proposition:MillerTreeHomeoKappaKappa} and Lemma~\ref{lemma: equivalent conditions for superperfect embedding}, $p[T]$ contains a closed subset 
 homeomorphic to ${}^\kappa\kappa$ in $\VV[G]$. 
\end{proof}


\section{The Hurewicz dichotomy at all regular cardinals}\label{section:Global}

This section is devoted to the proof of Theorem~\ref{theorem:GlobalHD}: the class forcing \( \vec{\PP} \) in its statement will be just an Easton support iteration of the partial order given by Theorem~\ref{theorem:Cons General}.
We will first concentrate on the forcing preservation statements listed in the theorem; actually, we will prove those statements for a larger and broad class of forcing iterations.

In the following, we say that a class-sized forcing iteration 
\[
\vec{\PP}  =  \langle\seq{\vec{\PP}_{{<}\alpha}}{\alpha\in\On},\seq{\dot{\PP}_\alpha}{\alpha\in\On}\rangle
\]  
is \emph{suitable} if it has Easton support and satisfies the following statements  for every $\nu\in\On$. 

 \begin{enumerate-(1)}
  \item If $\nu$ is an uncountable regular cardinal, then $\dot{\PP}_\nu$ is forced to be a ${<}\nu$-directed closed partial order satisfying the $\nu^+$-chain condition 
   that is a subset of $\HHH{\nu^+}$ and uniformly definable over $\langle\HHH{\nu^+},\in\rangle$. 

  \item If $\nu$ is not an uncountable regular cardinal, then $\dot{\PP}_\nu$ is the canonical $\vec{\PP}_{{<}\nu}$-name for the trivial forcing.  
 \end{enumerate-(1)}

\begin{lemma}\label{lemma:Class iteration propoerties}
 Assume that the $\mathsf{GCH}$ holds. Let $\vec{\PP}$ be a suitable iteration. 
 Given $\alpha\leq\beta$, we let $\dot{\PP}_{[\alpha,\beta)}$ denote the canonical $\vec{\PP}_{{<}\alpha}$-name for the corresponding tail forcing. 
 The following statements hold for every uncountable regular cardinal $\delta$. 
 \begin{enumerate-(1)}
  \item The partial order $\vec{\PP}_{{<}\delta}$ has cardinality at most $\delta$ and $\vec{\PP}_{{<}\delta+1}$ satisfies  the $\delta^+$-chain condition. 

  \item If $\alpha<\delta$, then $\vec{\PP}_{{<}\alpha}$ is a subset of $\HHH{\delta}$ that is uniformly definable over $\langle\HHH{\delta},\in\rangle$ in the parameter $\alpha$. 
   
  \item\label{statement:TailsClosed} If  $\alpha\geq\delta$, then $\mathbbm{1}_{\vec{\PP}_{{<}\delta+1}}\Vdash\anf{\text{$\dot{\PP}_{[\delta+1,\alpha)}$ is ${<}\check{\delta}^+$-directed closed}}$.
  
   \item\label{statement:CofinalityPreserving} Forcing with $\vec{\PP}$ preserves all cofinalities and the $\mathsf{GCH}$.   
 \end{enumerate-(1)}
\end{lemma}

\begin{proof}
 The first two statements inductively follow from our assumptions as in the proof of Theorem~\ref{theorem:Cons General}. 
 The third statement is a consequence of the first statement and {\cite[Proposition 7.12]{MR2768691}}. 
 
 Now, assume towards a contradiction, that forcing with $\vec{\PP}$ can change the cofinality of a regular cardinal. 
 Then there is a minimal cardinal $\nu$ such that forcing with $\vec{\PP}_{{<}\nu}$ can change such cofinalities. 
 By the first three statements, $\nu$ is a limit cardinal and forcing with $\nu$ preserves cofinalities greater than $(\nu^{{<}\nu})$ and less than or equal to $\nu$. 
 Hence $\nu$ is singular and forcing with $\vec{\PP}_{{<}\nu}$ can change the cofinality of $\nu^+$ to some regular cardinal $\bar{\nu}<\nu$. 
 By the third statement, this implies that forcing with $\vec{\PP}_{{<}\bar{\nu}+1}$ can change the cofinality of $\nu^+$, a contradiction.  
 
 Finally, let $\nu$ be an uncountable cardinal. Since $\vec{\PP}_{{<}\alpha}$ is $\sigma$-closed for every ordinal $\alpha$, forcing with $\vec{\PP}$ preserves $\mathsf{CH}$. 
 If $\nu$ is regular than the partial order $\vec{\PP}_{{<}\nu+1}$ has cardinality $\nu^+$ and satisfies the $\nu^+$-chain condition. 
 By the third statement, this shows that forcing with $\PP$ preserves the $\mathsf{GCH}$ at $\delta$. 
 Now, assume that $\nu$ is a singular cardinal. Let $G$ be $\vec{\PP}_{{<}\nu+1}$-generic over $\VV$ and $\bar{G}$ be the filter on $\vec{\PP}_{{<}\cof(\nu)}$ induced by $G$. 
 By the third statement, we have $({}^{\cof(\nu)}\nu)^{\VV[G]}\subseteq \VV[\bar{G}]$ and this implies 
 \[
(2^\nu)^{\VV[G]}=(\nu^{\cof(\nu)})^{\VV[G]}\leq(\nu^{\cof(\nu)})^{\VV[\bar{G}]}=(2^\nu)^{\VV[\bar{G}]}=(\nu^+)^{\VV[\bar{G}]}=(\nu^+)^{\VV[G]},
\] 
 because $\vec{\PP}_{{<}\bar{\nu}+1}$ has cardinality less than $\nu$ and forcing with $\vec{\PP}_{{<}\nu+1}$ preserves all cardinals.  
\end{proof}

In the following, we consider generalizations of weak compactness motivated by the definition of unfoldable cardinals in~\cite{MR1649079}. 
Remember that, given a regular cardinal $\delta$, a \emph{$\delta$-model} is a transitive model of $\ZFC^-$ of cardinality $\delta$ that is closed under ${<}\delta$-sequences.

\begin{definition}
 Let $\delta\leq\mu\leq\nu$ be cardinals with $\mu=\mu^{{<}\mu}$. We say that $\delta$ is \emph{$(\mu,\nu)$-supercompact} 
 if, for every $\mu$-model $M$ there is a transitive model $N$ and an elementary embedding $\map{j}{M}{N}$ with critical point $\delta$ such that ${}^\nu N\subseteq N$ and $j(\delta)>\nu$.  
\end{definition}

Note that, if the $\mathsf{GCH}$ holds, then we may assume that the model $N$ has cardinality $\nu^+$ in the above definition by repeatedly forming Skolem hulls and closures under $\nu$-sequences.

Given an ordinal $\theta$, we say that a cardinal $\delta\leq\theta$ is \emph{$\theta$-strongly unfoldable} if for every $\delta$-model $M$, there is a transitive set $N$ and an elementary embedding $\map{j}{M}{N}$ with critical point $\delta$  
such that $\VV_\theta\subseteq N$ and $j(\delta)>\theta$. A cardinal $\delta$ is \emph{strongly unfoldable} if it is $\theta$-strongly unfoldable for all $\theta\geq\delta$. 
Note that $\theta$-strongly unfoldable cardinals are weakly compact.

The following result due to D{\u{z}}amonja and Hamkins shows that we can use the above definition to characterize strongly unfoldable cardinals in models of the $\mathsf{GCH}$.

\begin{lemma}[{\cite[Lemma 5]{MR2279655}}] 
 A cardinal $\delta$ is $(\theta+1)$-strongly unfoldable if and only if it is $(\delta,\beth_\theta)$-supercompact. 
\end{lemma}

Remember that, given cardinals $\delta\leq\mu$, we say that $\delta$ is \emph{$\mu$-supercompact} if there is a transitive class $M$ and an elementary embedding $\map{j}{\VV}{M}$  
with critical point $\delta$ such that ${}^\mu M\subseteq M$ and $j(\delta)>\mu$. Moreover, the existence of such an embedding is equivalent to the existence of a normal ultrafilter on $\mathrm{P}_\kappa(\mu)$.

\begin{proposition}
 Let $\delta\leq\mu$ be cardinals with $\mu=\mu^{{<}\mu}$. If $\delta$ is $\mu$-supercompact, then it is $(\mu,\mu)$-supercompact. 
\end{proposition}

\begin{proof}
 Fix a $\mu$-model $M$. Let $\map{j}{\VV}{\WW}$ be a $\mu$-supercompact embedding.  
 Then $j(M)$ is a $j(\mu)$-model in $\WW$ and $j(\mu)>\mu$ implies $({}^\mu j(M))^\VV\subseteq ({}^\mu j(M))^\WW \subseteq j(M)$. 
 We can conlcude that $\map{j\restriction M}{M}{j(M)}$ is an elementary embedding with the desired properties.  
\end{proof}

Note that the above proof shows that $\delta$ is also $(\mu,\mu)$-supercompact in $\WW$.

\begin{proposition} 
 Let $\delta\leq\mu$ be cardinals with $\mu=\mu^{{<}\delta}$ and $2^\mu=\mu^+$. 
 If $\delta$ is $(\mu^+,\mu^+)$-supercompact, then it is $\mu$-supercompact. 
\end{proposition}

\begin{proof}
 Let $\theta$ be a sufficiently large regular cardinal and $X$ be an elementary submodel of $\HHH{\theta}$ of cardinality $\mu^+$ 
 such that ${}^\mu X\subseteq X$ and $\mu^++1\subseteq X$.  If we let $\map{\pi}{\langle X,\in\rangle}{\langle M,\in\rangle}$ 
 denote the corresponding transitive collapse, then $M$ is a $\mu^+$-model with $\mathrm{P}(\mathrm{P}_\delta(\mu))\subseteq M$. 
 By our assumption, there is a transitive model $N$ and an elementary embedding $\map{j}{M}{N}$ with ${}^{\mu^+}N\subseteq N$ and $j(\delta)>\mu^+$. 
 If we define \[
U  =  \Set{X\subseteq\mathrm{P}_\delta(\mu)}{j[\mu]\in j(X)},
\]
 then the above computations imply that $U$ 
 is a normal ultrafilter on $\mathrm{P}_\delta(\mu)$ in $\VV$. 
\end{proof}

Our next goal is to show that forcing with suitable forcing iteration preserves $(\mu,\nu)$-supercompact cardinals. 
In the proof of the preservation lemma, we make use of the notion of \emph{$\kappa$-proper forcing} introduced in {\cite{MR2342475}} (see {\cite[Definition 12]{MR2467213}})  
and related results proved in \cite[Section 4]{MR2467213}. In addition to these results, we need the following lemma.

\begin{lemma}\label{lemma:TwoStepKappaProper}
 Let $\delta$ be an uncountable cardinal with $\delta=\delta^{{<}\delta}$. 
 If $\mathbb{P}$ is a partial order satisfying the $\delta^+$-chain condition and $\dot{\QQ}$ is a $\PP$-name for a ${<}\delta^+$-closed partial order, 
 then the partial order $\PP*\dot{\QQ}$ is $\delta$-proper. 
\end{lemma} 

\begin{proof} 
 Fix a sufficiently large regular cardinal $\theta$ and an elementary submodel $X$ of $\HHH{\theta}$ of cardinality $\delta$ with ${}^{{<}\delta}X\subseteq X$, $\delta+1\subseteq X$ and $\PP*\dot{\QQ}\in X$. 
 Pick $\langle p_0,\dot{q}_0\rangle\in(\PP*\dot{\QQ})\cap X$. 
 Let $\seq{\dot{D}_\alpha}{\alpha<\delta}$ be an enumeration of all $\PP$-names for dense subsets of $\dot{\mathbb{Q}}$ contained in $X$. 
 By our assumptions, there is a sequence $\seq{\dot{q}_\alpha}{\alpha<\delta}$ of $\PP$-names for conditions in $\dot{\QQ}$ contained in $X$ such that 
 $\mathbbm{1}_\PP\Vdash\anf{\dot{q}_{\alpha+1}\in \dot{D}_\alpha}$ and $\mathbbm{1}_\PP\Vdash\anf{\dot{q}_\beta\leq_{\dot{\QQ}} \dot{q}_\alpha}$ holds for all $\alpha<\beta<\delta$. 
 Our assumption on $\dot{\QQ}$ imply that there is a $\PP$-name $\dot{q}$ for a condition in $\dot{\QQ}$ with $\mathbbm{1}_\PP\Vdash\anf{\dot{q}\leq_{\dot{\QQ}}\dot{q}_\alpha}$ for all $\alpha<\delta$. 
 In the following, we will show that the condition $\langle p_0,\dot{q}\rangle$ is $(X,\PP*\dot{\QQ})$-generic.

 Suppose that $D\in X$ is a dense subset of $\PP*\dot{\QQ}$. Let $\seq{\langle r_{\beta},\dot{s}_{\beta})}{\beta<\rho}$ be an enumeration of $D$ in $X$.  
 Define $\dot{D}=\Set{\langle\dot{s}_{\beta},r_{\beta}\rangle}{\beta<\rho}\in X$. Then $\dot{D}$ is a $\PP$-name.

 \begin{claim} 
  $\mathbbm{1}_\PP\Vdash\anf{\text{$\dot{D}$ is a dense subset of $\dot{\QQ}$}}$. 
 \end{claim} 

 \begin{proof}[Proof of the Claim]
  Let $p$ be a condition in $\PP$ and $\dot{q}$ be a name for a condition in $\dot{\QQ}$. Then $\langle p,\dot{q}\rangle$ is a condition in $\PP*\dot{\QQ}$ and there is 
  a $\beta<\rho$ with $\langle r_\beta,\dot{s}_\beta\rangle\leq_{\PP*\dot{\QQ}}\langle p,\dot{q}\rangle$. 
  Then $r_\beta\leq_\PP p$, $r_\beta\Vdash\anf{\dot{s}_\beta\leq_{\dot{\QQ}} \dot{q}}$, and $r_\beta\Vdash\anf{\dot{s}_\beta\in \dot{D}}$. 
  This shows that the set $\bar{D}=\Set{r\in\PP}{r\Vdash\anf{\exists s\in\dot{D} \, ( \dot{s}\leq_{\dot{\QQ}}\dot{q})}}$ is dense in $\PP$ and this yields the statement of the claim. 
 \end{proof}

 By our construction, there is an $\alpha<\delta$ with  $\dot{D}=\dot{D}_{\alpha}$.

 \begin{claim} 
  The set 
   \[
E  =  \lrSet{p\in\PP}{\exists\beta<\rho \, \left({p\leq_\PP r_\beta}  \, \wedge \,  {p\Vdash\anf{\dot{q}_\alpha=\dot{s}_\beta}}\right)}\in X
\] 
  is dense open below $p_0$ in $\PP$. 
 \end{claim} 

 \begin{proof}[Proof of the Claim]
  Suppose that $p_1\leq_\PP p_0$ and let $H$ be $\PP$-generic over $V$ with $p_1\in H$. 
  Since $\dot{q}_\alpha^H\in\dot{D}^H$, there is a $\beta<\rho$ with $r_\beta\in H$ and $\dot{q}_{\alpha}^H=\dot{s}_\beta^H$. 
  Hence, there is a condition in $E\cap H$ below $p_1$. 
 \end{proof}

 Let $A\in X$ be a maximal antichain in $E$ below $p_0$. Since $\PP$ satisfies that $\delta^+$-chain condition, we have $A\subseteq X$. 
 Suppose that $G=G_0*G_1$ is $\PP*\dot{\QQ}$-generic over $V$ with $\langle p_0,\dot{q}\rangle\in G$. By the above remark, there is a condition $p\in A\cap G_0\cap X$. 
 Then $\langle p,\dot{\mathbbm{1}}_{\dot{\QQ}})\in G$ and hence $\langle p,\dot{q}\rangle\in G$, because $p\leq_\PP p_0$ and $\langle p_0,\dot{q}\rangle\in G$. 
 Since $\mathbbm{1}_{\PP}\Vdash\anf{\dot{q}\leq_{\dot{\QQ}}\dot{q}_\alpha}$, this implies that $\langle p,\dot{q}_{\alpha}\rangle\in G$. 
 We have  $p\in A\subseteq E\cap X$ and this implies that there is a $\beta\in X\cap \delta$ with $p\leq_\PP r_\beta$ and $p\Vdash\anf{\dot{q}_{\alpha}=\dot{s}_{\beta}}$. 
 We can conclude that  $\langle p,\dot{q}_{\alpha}\rangle\leq_{\PP*\dot{\QQ}}\langle r_{\beta},\dot{s}_\beta\rangle\in X$ and thus $\langle r_{\beta},\dot{s}_\beta\rangle \in D\cap G\cap X$.  
\end{proof}

We are now ready to prove our large cardinal preservation result.

\begin{lemma}
 Assume that the $\mathsf{GCH}$ holds. Let $\vec{\PP}$ be a suitable iteration. 
 If $\delta$ is a $(\mu,\nu)$-supercompact cardinal with $\nu$ regular, then $\delta$ remains $(\mu,\nu)$-supercompact in every $\vec{\PP}$-generic extension of the ground model. 
\end{lemma}

\begin{proof}
 Let $G$ be $\vec{\PP}_{{<}\nu+1}$-generic over $\VV$, $M_0$ be a $\mu$-model in $\VV[G]$ and $\dot{x}$ be a $\vec{\PP}_{{<}\nu+1}$-name for a subset of $\mu$ coding $M_0$. 
 Work in $\VV$ and fix a sufficiently large regular cardinal $\theta$. 
 We know that $\vec{\PP}_{{<}\nu+1}$ densely embeds into $\vec{\PP}_{{<}\mu+1}*\dot{\PP}_{[\mu+1,\nu+1)}$ and, by  Lemma~\ref{lemma:Class iteration propoerties}, Lemma~\ref{lemma:TwoStepKappaProper} and {\cite[Fact 13]{MR2467213}}, 
 this implies that the partial order  $\vec{\PP}_{{<}\nu+1}$  is $\mu$-proper. 
 This allows us to find an elementary submodel $X$ of $\HHH{\theta}$ of cardinality $\mu$ and conditions $\vec{p},\vec{q}\in\vec{\PP}_{{<}\nu+1}$ 
 such that ${}^{{<}\mu}X\subseteq X$, $\mu+1\subseteq X$, $\vec{p},\dot{x},\vec{\PP}_{{<}\nu+1}\in X$, $\vec{q}\leq_{\vec{\PP}_{{<}\nu+1}}\vec{p}$ 
 and $\vec{q}\in G$ is $(X,\vec{\PP}_{{<}\nu+1})$-generic. 
 Let $\map{\pi}{\langle X,\in\rangle}{\langle M,\in\rangle}$ denote the corresponding transitive collapse. 
 Then $M$ is a $\mu$-model and our assumptions imply that there is a transitive model $N$ of cardinality $\nu^+$ and an elementary embedding $\map{j}{M}{N}$ 
 such that ${}^\nu N\subseteq N$ and $j(\delta)>\nu$.

 Let $\vec{\PP}_{\nu+1}$ denote the restriction of $\vec{\PP}$ to $\nu+1$. 
 We define $\delta_M=\pi(\delta)$, $\delta_N=j(\delta_M)$, $\nu_M=\pi(\nu)$ and $\nu_N=j(\nu_M)$. Moreover, set 
 \[
\vec{\QQ}  =   \pi(\vec{\PP}_{\nu+1})  =  \langle\seq{\vec{\QQ}_{{<}\alpha}}{\alpha\leq\nu_M+1},\seq{\dot{\QQ}_\alpha}{\alpha\leq\nu_M}\rangle
\]
 and 
\[
\vec{\RR}  =   j(\vec{\QQ})  =  \langle\seq{\vec{\RR}_{{<}\alpha}}{\alpha\leq\nu_N+1},\seq{\dot{\RR}_\alpha}{\alpha\leq\nu_N}\rangle.
\]
 By our assumptions, we have $\HHH{\mu}^\VV=\HHH{\mu}^M=\HHH{\mu}^N$ and $\HHH{\nu^+}^\VV=\HHH{\nu^+}^N$. 
 This implies that $\vec{\PP}_{{<}\mu}=\vec{\QQ}_{{<}\mu}=\vec{\RR}_{{<}\mu}$  
 and $\vec{\PP}_{{<}\nu+1}=\vec{\RR}_{{<}\nu+1}$.

 \begin{claim}
  $({}^\nu N[G])^{\VV[G]}\subseteq N[G]$.
 \end{claim}

 \begin{proof}[Proof of the Claim]
  Let $x$ be a subset of $\nu$ in $\VV[G]$. Then there is a $\vec{\PP}_{{<}\nu+1}$-nice name $\dot{x}$ for $x$ in $\VV$. 
  Since $\vec{\PP}_{{<}\nu+1}$ satisfies the $\nu^+$-chain condition in $\VV$, our assumptions imply $\dot{x}\in\HHH{\nu^+}^\VV\subseteq N$ and this shows that $x$ is an element of $N[G]$. 
  Now, pick $x\in([N\cap\On]^\nu)^{\VV[G]}$. Since $\vec{\PP}_{{<}\nu+1}$ satisfies the $\nu^+$-chain condition in $\VV$, there are $f,y\in\VV$ with $x\subseteq y\subseteq N\cap\On$ and $\map{f}{y}{\nu}$ is a bijection. 
  Our assumptions imply $f,y\in N$. By the above computations, $f[x]\subseteq\nu$ is an element of $N[G]$ and this implies $x\in N[G]$.  
  These computations show that ${}^\nu(N\cap \On)\subseteq N[G]$ and this yields the statement of the claim, because $N[G]$ is a transitive model of $\ZFC^-$. 
 \end{proof}

  By the above computations, the partial order $\RR_0=\dot{\RR}_{[\nu+1,\delta_N)}^G$ is ${<}\nu^+$-closed in $\VV[G]$. 
  Moreover, its power set in $N[G]$ has cardinality at most $\nu^+$ in $\VV[G]$. 
  This shows that there is an $F\in\VV[G]$ that is $\RR_0$-generic over $N[G]$. 
  Let $H_0$ denote the filter on $\vec{\RR}_{{<}\delta_N}$ induced by $G$ and $F$. 
  By the above claim, we have $({}^\nu N[H_0])^{\VV[G]}\subseteq N[H_0]$.

  By the definition of $\dot{\PP}_\delta$, Lemma~\ref{lemma:Class iteration propoerties} and the above computations, 
  we know that $\RR_1=\dot{\RR}_{[\delta_N,\nu_N+1)}^{H_0}$ is ${<}\nu^+$-directed closed in $\VV[G]$. 
  Define $\bar{G}=\pi[G\cap X]$. Since $\vec{q}\in G$ is $X$-generic, we know that $\bar{G}$ is $\vec{\QQ}_{{<}\nu_M+1}$-generic over $M$. 
  Moreover,   
\[
D  =  \Set{\vec{r}\restriction[\delta_N,\nu_N+1)}{\vec{r}\in j[\bar{G}}
\]
  is a directed subset of $\RR_1$ of cardinality $\nu$ and we can conclude that there is a condition $\vec{m}$ in $\RR_1$ such that $\vec{m}\leq_{\RR_1}\vec{r}$ holds for all $\vec{r}\in D$.  
  The power set of $\RR_1$ in $N[H_0]$ has cardinality  at most $\nu^+$ in $\VV[G]$ and there is  $H_1\in\VV[G]$ that is $\RR_1$-generic over $N[H_0]$ with $\vec{m}\in H_1$.  
  Let $H$ be the filter on $\vec{\RR}_{{<}\nu_N+1}$ induced by $H_0$ and $H_1$. Then another application of the above claim yields $({}^\nu N[H])^{\VV[G]}\subseteq N[H]$.

  Pick $\vec{r}\in\bar{G}$. Since the support of $\vec{r}$ in $\vec{\QQ}_{{<}\nu_M+1}$ below $\delta$ is bounded, we have $j(\vec{r})\restriction\delta=\vec{r}\restriction\delta$, $j(\vec{r})\restriction[\delta,\delta_N)$ is trivial and 
  $j(\delta)\restriction[\delta_N,\nu_N+1)\in D$. This shows that $j(\vec{r})\in H$. In this situation, we can apply {\cite[Proposition 9.1]{MR2768691}} to construct an elementary embedding $\map{j_*}{M[\bar{G}]}{N[H]}$ extending $j$.  
  By {\cite[Lemma 19]{MR2467213}}, we know that $\dot{x}^G=\pi(\dot{x})^{\bar{G}}\in M[\bar{G}]$ and this implies that $M_0$ is a $\mu$-model in $M[\bar{G}]$. 
  By elementarity, $j(M_0)$ is a $j(\mu)$-model in $N[H]$ and $j(\mu)\geq j(\kappa)>\nu$ implies that $({}^\nu j(M_0))^{\VV[G]}\subseteq ({}^\nu j(M_0))^{N[H]}\subseteq j(M_0)$. 
  We can conclude that $\map{j_*\restriction M_0}{M_0}{j(M_0)}$ is an embedding with the desired properties. 

 The above argument shows that $\kappa$ remains $(\mu,\nu)$-supercompact in every $\vec{\PP}_{{<}\nu+1}$-generic extension. 
 Since $\dot{\PP}_{[\nu+1,\alpha)}$ is a forced to be ${<}\nu^+$-closed for every $\alpha>\nu$, this conclusion implies the statement of the lemma.  
\end{proof}

\begin{corollary}
 Assume that $\mathsf{GCH}$ holds. If $\vec{\PP}$ is a suitable iteration, then forcing with $\PP$ preserves strongly unfoldable and supercompact cardinals.  
\end{corollary}

We are now ready to prove Theorem~\ref{theorem:GlobalHD} with the help of the above results and Theorem~\ref{theorem:Cons General}.

\begin{proof}[Proof of Theorem~\ref{theorem:GlobalHD}] 
 Assume that the $\mathsf{GCH}$ holds. 
 We define a class-sized forcing  iteration  $\vec{\PP}  =  \langle\seq{\vec{\PP}_{{<}\alpha}}{\alpha\in\On},\seq{\dot{\PP}_\alpha}{\alpha\in\On}\rangle$ with Easton support by the following clauses for all $\nu\in\On$.  

 \begin{enumerate-(a)}
  \item If $\mathbbm{1}_{\vec{\PP}_{{<}\nu}}\Vdash\anf{\text{$\check{\nu}$ is an uncountable cardinal with $\check{\nu}=\check{\nu}^{{<}\check{\nu}}$}}$, 
   then $\dot{\PP}_\nu$ is a $\vec{\PP}_{{<}\nu}$-name for the partial order $\PP(\nu)$ given by Theorem~\ref{theorem:Cons General}.  
  \item Otherwise, $\dot{\PP}_\nu$ is a $\vec{\PP}_{{<}\nu}$-name for the trivial forcing.  
 \end{enumerate-(a)}

 It follows directly from the properties of $\PP(\nu)$ listed in Theorem~\ref{theorem:Cons General} that the iteration $\vec{\PP}$ is suitable. 
 Thus the above results show that $\vec{\PP}$ satisfies condition~\ref{theorem:GlobalHD-1} of Theorem~\ref{theorem:GlobalHD}. 
 Let now $\VV[G]$ be a $\vec{\PP}$-generic extension of $\VV$, $\kappa$ be an uncountable regular cardinal and $\bar{G}$ be the filter on $\vec{\PP}_{{<}\kappa+1}$ induced by $G$. 
 By the above remarks and  Theorem~\ref{theorem:Cons General}, every $\mathbf{\Sigma}^1_1$ subset of $\pre{\kappa}{\kappa}$ satisfies the Hurewicz dichotomy in $\VV[\bar{G}]$. 
 Since Lemma~\ref{lemma:Class iteration propoerties} implies that $\HHH{\kappa^+}^{\VV[G]}=\HHH{\kappa^+}^{\VV[\bar{G}]}$, this shows that condition~\ref{theorem:GlobalHD-2} of Theorem~\ref{theorem:GlobalHD} holds as well. 
\end{proof}


\section{Connections with regularity properties} \label{sec:regularityproperties}

In this section we study the mutual relationships between the (various form of the) Hurewicz dichotomy, the \( \kappa \)-perfect set property, and some other regularity properties. Our starting point is the following simple observation.

\begin{lemma}\label{lemma:MeasureabilityFromHD}
 Assume that all $\mathbf{\Sigma}^1_1$ subsets of ${}^\kappa\kappa$ satisfy the Hurewicz dichotomy. If $A$ is a $\mathbf{\Sigma}^1_1$ subset of ${}^\kappa\kappa$ and $C$ is closed subset of ${}^\kappa\kappa$ homeomorphic to ${}^\kappa\kappa$, 
 then there is a closed subset $D$ of $C$ homeomorphic to ${}^\kappa\kappa$ such that either $D\subseteq A$ or $A\cap D=\emptyset$.  Hence the same conclusion is true for \( \mathbf{\Pi}^1_1 \) sets \( A \subseteq \pre{\kappa}{\kappa} \) as well.
\end{lemma}

\begin{proof}
 Assume that $A\cap C$ does not contain a closed-in-\( \pre{\kappa}{\kappa} \) subset homeomorphic to the whole ${}^\kappa\kappa$. 
 Since by our hypothesis $A\cap C$ must satisfy the Hurewicz dichotomy, there is a $\K{\kappa}$ set $K \subseteq {}^\kappa\kappa$ with 
$A\cap C\subseteq K$.  Consider now the set $C\setminus K$. Since it is a \( \mathbf{\Sigma}^1_1 \) set, it satisfies the Hurewicz dichotomy as well. But \( C \) cannot be 
contained in a $\K{\kappa}$ set by Fact~\ref{proposition:DichotomyUnc}, hence we can conclude that $C\setminus K$ contains a 
closed-in-\( \pre{\kappa}{\kappa} \) subset $D$ homeomorphic to ${}^\kappa\kappa$. Since 
$D \subseteq {C \setminus K} \subseteq {C\setminus A}$, we are done. 
\end{proof}

Recall from Definition~\ref{def:regularityproperties} the notions of \( \kappa \)-Sacks measurability and \( \kappa \)-Miller measurability.

\begin{theorem}\label{theorem:HDimpliesMeasurability}
 Let $\kappa$ be an uncountable cardinal with $\kappa=\kappa^{{<}\kappa}$. Assume that every $\mathbf{\Sigma}^1_1$ subset of ${}^\kappa\kappa$ satisfies the Hurewicz dichotomy. 
 \begin{enumerate-(1)}
  \item \label{theorem:HDimpliesMeasurability-1}
If $\kappa$ is not weakly compact, then all $\mathbf{\Sigma}^1_1$ and all \( \mathbf{\Pi}^1_1 \) subsets of ${}^\kappa\kappa$ are $\kappa$-Sacks measurable. 
  
  \item \label{theorem:HDimpliesMeasurability-2}
If $\kappa$ is weakly compact, then all $\mathbf{\Sigma}^1_1$ and all \( \mathbf{\Pi}^1_1 \) subsets of ${}^\kappa\kappa$ are $\kappa$-Miller measurable. 
 \end{enumerate-(1)}
\end{theorem}

\begin{proof}
It is clearly enough to consider the case of \( \mathbf{\Sigma}^1_1 \) sets, so let $A \subseteq \pre{\kappa}{\kappa}$ be $\mathbf{\Sigma}^1_1$ and $T \subseteq {}^{{<}\kappa}\kappa$ be a subtree. 

\ref{theorem:HDimpliesMeasurability-1} Assume that $\kappa$ is not weakly compact and that $T$ is perfect. Then $[T]$ contains a closed subset $C$ homeomorphic to ${}^\kappa 2$ by Lemma~\ref{lemma:PerfectClosedTreeEmb}, and 
 Corollary~\ref{corollary:RevNegroPonte} implies that $C$ is homeomorphic to ${}^\kappa\kappa$. By Lemma~\ref{lemma:MeasureabilityFromHD} there is a closed subset $D$ of $C$ homeomorphic to ${}^\kappa\kappa$ 
 such that either $D\subseteq A$ or $A\cap D=\emptyset$. Using Lemma~\ref{lemma:PerfectClosedTreeEmb} again, we get that there is a perfect tree $S$ with $[S]\subseteq D$. 
 Since $[S]\subseteq [T]$ holds and every node in $S$ is extended by a $\kappa$-branch through $S$, we can conclude that $S$ is a perfect subtree of $T$ such that either $[S]\subseteq A$ or $A\cap[S]=\emptyset$, as required.

\ref{theorem:HDimpliesMeasurability-2} Assume that $\kappa$ is weakly compact and that $T$ is a $\kappa$-Miller tree. By Proposition~\ref{proposition:MillerTreeHomeoKappaKappa} the set $[T]$ is homeomorphic to ${}^\kappa\kappa$, hence 
 Lemma~\ref{lemma:MeasureabilityFromHD} yields a closed subset $D$ of $[T]$ such that either $D\subseteq A$ or $A\cap D=\emptyset$. 
 By Proposition~\ref{proposition:MillerTreeHomeoKappaKappa} there is $\kappa$-Miller tree $S$ with $[S]\subseteq D$, and hence arguing as above we can conclude that  $S$ is a \( \kappa \)-Miller subtree of \( T \) such that either $[S]\subseteq A$ or $A\cap[S]=\emptyset$, as required.  
\end{proof}

In the remainder of this section we use the forcing construction from Section~\ref{section:Global} to separate the Hurewicz dichotomy, the \( \kappa \)-Sacks measurability, and the \( \kappa \)-Miller measurability from the \( \kappa \)-perfect set property.

\begin{theorem}\label{theorem:SeparatePSPfromHD}
 Assume $\VV=\LL$ and let $\vec{\PP}$ be the class forcing constructed in the proof of Theorem~\ref{theorem:GlobalHD}. 
 If $\kappa$ is an uncountable regular cardinal, then there is a subtree $T$ of ${}^{{<}\kappa}\kappa$ such that in every $\vec{\PP}$-generic extension of $\VV$, the closed set $[T]$ does not have the \( \kappa \)-perfect set property. 
\end{theorem} 

\begin{proof}
 Fix an uncountable regular cardinal $\kappa$ and assume that $G$ is $\vec{\PP}_{{<}\kappa+1}$-generic over $\VV$. 

 \begin{claim} \label{claim:slim}
  There are $S,T\in\VV$ such that $S$ is a stationary subset of $\kappa$ in $\VV[G]$ and $T$ is a \emph{weak $S$-Kurepa tree} in $\VV[G]$, i.e.\  $T$ is a subtree of ${}^{{<}\kappa}\kappa$ 
  such that $[T]$ has size $\kappa^+$ and $\vert T\cap {}^\alpha \kappa	\vert\leq\vert\alpha\vert$ holds for all $\alpha\in S$. 
 \end{claim}

 \begin{proof}[Proof of the Claim]
  First assume that $\kappa$ is not ineffable in $\VV$ (see {\cite[Section 2]{JensenKunen1969:Ineffable}}). By the results of~\cite{JensenKunen1969:Ineffable}, this implies that there is a \emph{slim $\kappa$-Kurepa tree}, i.e.\  a subtree $T$ of ${}^{{<}\kappa}\kappa$ with $[T]$ of size $\kappa^+$ and $\vert T\cap {}^\alpha \kappa	\vert\leq\vert\alpha\vert$ for all $\alpha<\kappa$. Since forcing with $\vec{\PP}_{{<}\kappa+1}$ 
  preserves cardinalities, $T$ is a slim $\kappa$-Kurepa tree in $\VV[G]$ as well. 

  Now assume that $\kappa$ is ineffable in $\VV$. Then $\kappa$ is a Mahlo cardinal in $\VV$ and {\cite[Proposition 7.13]{MR2768691}} shows that $\vec{\PP}_{{<}\kappa}$ satisfies the $\kappa$-chain condition. 
  Since $\dot{\PP}_\kappa$ is a $\vec{\PP}_{{<}\kappa}$-name for a ${<}\kappa$-closed partial order, this implies that forcing with $\vec{\PP}_{{<}\kappa+1}$ preserves stationary subsets of $\kappa$ and all cardinalities. 
  By our assumptions, we can apply the results of {\cite[Section 3]{FriedmanKulikov}} to find a stationary subset $S$ of $\kappa$ and a weak $S$-Kurepa tree in $\VV$. 
  Thus by the above remarks the statement of the claim holds also in this case. 
 \end{proof}

 Let $T$ be a subtree as in Claim~\ref{claim:slim}. An easy argument (see for example {\cite[Section 7]{PL}}) shows that the closed set $[T]$ does not have the \( \kappa \)-perfect set property in $\VV[G]$. Hence by Lemma~\ref{lemma:Class iteration propoerties} we are done. 
\end{proof}

Recall from Proposition~\ref{prop:PSPHurewicz} that when \( \kappa \) is not weakly compact the \( \kappa \)-perfect set property implies (pointwise) the Hurewicz dichotomy. However, the next corollary shows that the latter is really weaker: in fact we can separate the two properties already at the level of closed subsets of \( \pre{\kappa}{\kappa} \).

\begin{corollary} \label{cor:HDnonPSP}
It is consistent with \( \ZFC \) that for every uncountable regular cardinal \( \kappa \), all \( \mathbf{\Sigma}^1_1 \) subsets of \( \pre{\kappa}{\kappa} \) satisfy the Hurewicz dichotomy while there are closed subsets of \( \pre{\kappa}{\kappa} \) without the \( \kappa \)-perfect subset property.
\end{corollary}

\begin{proof}
Force over \( \LL \) with the class forcing $\vec{\PP}$ constructed in the proof of Theorem~\ref{theorem:GlobalHD}: then the result follows from Theorems~\ref{theorem:GlobalHD} and~\ref{theorem:SeparatePSPfromHD}.
\end{proof}

As observed in the introduction, the \( \kappa \)-perfect set property also implies (pointwise) its symmetric version, namely the \( \kappa \)-Sacks 
measurability. The following result allows us to separate these two properties at  the level of closed subsets of \( \pre{\kappa}{\kappa} \) (for a 
non-weakly compact cardinal \( \kappa \)). Moreover, it shows that we can separate also the \( \kappa \)-Miller measurability from the 
\( \kappa \)-perfect set property for \( \kappa \) a weakly compact cardinal.

\begin{corollary} \label{cor:measurabilitynonPSP}
It is consistent with \( \ZFC \) that for every non-weakly compact regular cardinal \( \kappa \) all \( \mathbf{\Sigma}^1_1 \) and all \( \mathbf{\Pi}^1_1 \) subsets of \( \pre{\kappa}{\kappa} \) are \( \kappa \)-Sacks measurable, for every weakly compact cardinal \( \kappa \) all \( \mathbf{\Sigma}^1_1 \) and all \( \mathbf{\Pi}^1_1 \) subsets of \( \pre{\kappa}{\kappa} \) are \( \kappa \)-Miller measurable, while for every uncountable regular cardinal $\kappa$ there are closed subsets of \( \pre{\kappa}{\kappa} \) without the \( \kappa \)-perfect subset property.
\end{corollary}

\begin{proof}
Force again over \( \LL \) with the class forcing $\vec{\PP}$ constructed in the proof of Theorem~\ref{theorem:GlobalHD}: the result then follows from Theorems~\ref{theorem:GlobalHD} and~\ref{theorem:HDimpliesMeasurability}. 
\end{proof}

Notice that  if there are strongly unfoldable cardinals in \( \LL \), then by Theorem~\ref{theorem:GlobalHD} such cardinals remain strongly unfoldable (hence also weakly compact) in the generic extension of \( \LL \) considered in the above proof. Thus our results show that 
the possibility of separating the \( \kappa \)-Miller measurability from the \( \kappa \)-perfect set property (for some uncountable cardinal \( \kappa \))
 is consistent relatively to some large cardinal assumption: we do not know if such separation can be obtained from \( \ZFC \) alone, as it is already the case for the \( \kappa \)-Sacks measurability.


\section{Failures of the Hurewicz dichotomy}\label{section:FailureCohen}

We show that the Hurewicz dichotomy fails above a regular cardinal $\mu$ after we added a single Cohen subset of $\mu$ to the ground model.

\begin{theorem}\label{theorem:CounterexampleCohenReal}
Suppose that $\mu=\mu^{<\mu}$ and $G$ is $\Add{\mu}{1}$-generic over $\VV$. If $\kappa>\mu$ is a cardinal with $\kappa=\kappa^{{<}\kappa}$, then $({}^{\kappa}\kappa)^\VV$ is a closed subset of ${}^\kappa\kappa$ in $\VV[G]$ that does not satisfy the Hurewicz dichotomy. 
\end{theorem}

We prove this result with the help of the following lemmas.

\begin{lemma}\label{covering kappa-compact sets in generic extensions} 
 Assume that $\PP$ is a partial order satisfying the $\kappa$-chain condition such that forcing with $\PP$ preserves all $\kappa$-Aronszajn trees.  
 Let $\dot{T}$ be a $\PP$-name for a pruned subtree of ${}^{{<}\kappa}\kappa$ such that 
 \[
\mathbbm{1}_\PP \Vdash\anf{\text{$\dot{T}\subseteq \dot{\VV}$ and $[\dot{T}]$ is $\kappa$-compact}}.
\]
 Then there is a pruned subtree $T$ of ${}^{{<}\kappa}\kappa$ such that the closed set $[T]$ is $\kappa$-compact and $\mathbbm{1}_\PP \Vdash \anf{[\dot{T}]\subseteq[\check{T}]}$. 
\end{lemma} 

\begin{proof} 
 We define 
\[T  =  \Set{t\in {}^{{<}\kappa}\kappa}{\exists p\in\PP \, (  p\Vdash\anf{ \check{t}\in \dot{T}} )}.
\] 
 Then $T$ is a subtree of ${}^{{<}\kappa}\kappa$ of height $\kappa$ with $\mathbbm{1}_\PP\Vdash\anf{\dot{T}\subseteq\check{T}}$. 
 Moreover, for every $s\in T$ and every $\alpha<\kappa$, there is a $t\in T$ with $s\subseteq t$ and $\length{t}>\alpha$.

 \begin{claim}
  $T$ is a $\kappa$-tree.
 \end{claim}
 
 \begin{proof}[Proof of the Claim]
  Assume, towards a contradiction, that there is an $\alpha<\kappa$ with $\vert T\cap{}^\alpha\kappa\vert=\kappa$. 
  Let $\seq{t(\alpha)}{\alpha<\kappa}$ be an enumeration of $T\cap{}^\alpha\kappa$. 
  By Lemma~\ref{lemma:kurepa}, the set $D$ of all $p\in\PP$ such that there is an $\beta_p<\kappa$ with 
  \[
p \Vdash \anf{\forall t\in \dot{T}\cap{}^{\check{\alpha}}\check{\kappa} \; \exists\beta<\check{\beta}_p \, (t=\check{t}(\beta))}
\] 
is dense in $\PP$. 
  Let $A$ be a maximal antichain in $D$. Since $\PP$ satisfies the $\kappa$-chain condition, we have $\beta_*=\sup\Set{\beta_p}{p\in A}<\kappa$. 
  But this implies $\mathbbm{1}_\PP\Vdash\anf{\check{t}(\check{b}_*+1)\notin\dot{T}}$, a contradiction. 
 \end{proof}

 \begin{claim}
  $T$ is pruned. 
 \end{claim}
 
 \begin{proof}[Proof of the Claim]
  Assume, towards a contradiction, that there is a $t\in T$ with $N_t\cap[T]=\emptyset$. Define $T_t=\Set{s\in T}{s\subseteq t\vee t\subseteq s}$. 
  By the above remarks, $T_t$ is a $\kappa$-tree and our assumption implies that $T_t$ is a $\kappa$-Aronszajn tree. 
  Pick $p\in\PP$ with $p\Vdash\anf{\check{t}\in \dot{T}}$ and let $G$ be $\PP$-generic over $\VV$ with $p\in G$. 
  Since $\dot{T}^G$ is pruned in $\VV[G]$ and $t\in\dot{T}^G$, there is an $x\in[\dot{T}^G]^{\VV[G]}$ with $t\subseteq x$.
  This implies $x\in[T_t]^{\VV[G]}$ we can conclude that $T_t$ is not a $\kappa$-Aronszajn tree in $\VV[G]$, a contradiction. 
 \end{proof}

 \begin{claim}
  $T$ does not contain a $\kappa$-Aronszajn subtree. 
 \end{claim}
 
 \begin{proof}[Proof of the Claim]
  Assume, towards a contradiction, that there is a $\kappa$-Aronszajn subtree $S$ of $T$. 
  By our assumptions, the set $D$ of all $p\in\PP$ such that there is an $\beta_p<\kappa$ with 
  \[
p \Vdash \anf{\forall t\in \check{S}\cap\dot{T} \, ( \length{t}<\check{\beta}_p)}
\]
 is dense in $\PP$. 
  Let $A$ be a maximal antichain in $D$. Since $\PP$ satisfies the $\kappa$-chain condition, 
  there is an $s\in S$ with $\length{s}>\beta_p$ for all $p\in A$. This implies that $\mathbbm{1}_\PP\Vdash\anf{\check{t}\notin\dot{T}}$, a contradiction.   
 \end{proof}

 Our definition directly implies $\mathbbm{1}_\PP \Vdash \anf{[\dot{T}]\subseteq[\check{T}]}$ and Lemma~\ref{lemma:kurepa} shows that the closed set $[T]$ is $\kappa$-compact.  
\end{proof}

Note that partial orders of cardinality less than $\kappa$ satisfy the assumptions of the above lemma.  
We sketch an argument showing that, in general, the second assumptions of the lemma cannot be omitted: 
Suppose that $G$ is $\Add{\kappa}{1}$-generic over $\VV$ and $\kappa$ is weakly compact in $\VV[G]$. By a classical argument of Kunen (see~\cite{MR495118}), 
there is an intermediate model $\VV\subseteq\WW\subseteq\VV[G]$ such that there is a homogenous $\kappa$-Souslin tree $S\subseteq{}^{{<}\kappa}\kappa$ in $\WW$ and $\VV[G]$ is a forcing extension of $\WW$ 
by the canonical forcing that satisfies the $\kappa$-chain condition and adds a $\kappa$-branch through $S$. 
The homogeneity of $S$ in $\WW$ implies that every node in $S$ is extended by a $\kappa$-branch through $S$ in $\VV[G]$ and hence $S$ is pruned in $\VV[G]$.  
Moreover, $S$ is a $\kappa$-tree without $\kappa$-Aronszajn subtrees and this shows that  the set $[S]$ is $\kappa$-compact in $\VV[G]$.   
Assume, towards a contradiction, that there is a pruned subtree $T$ of ${}^{{<}\kappa}\kappa$ in $\WW$ such that $[T]$ is $\kappa$-compact in $\WW$ and $[S]^{\VV[G]}\subseteq[T]^{\VV[G]}$.  By the above remarks, we have $S\subseteq T$ and we can conclude that $[T]$ is not $\kappa$-compact in $\WW$, a contradiction.


The next lemma deals with the question whether it is possible to have an infinite cardinal $\mu$ and an inner model $M$ such that $M$ contains all branches through a perfect subtree of ${}^{{<}\mu}\mu$ and there is a subset of $\mu$ that is not contained in $M$. 
This question is closely connected to questions of Prikry (see {\cite[Question 1.1]{MR1632148}}) and Woodin (see {\cite[Question 3.2]{MR1632148}}) and the results of~\cite{MR1632148} and~\cite{MR1640916}. 
The following lemma shows that it is not possbile to obtain a positive answer to the above question by adding Cohen-subsets to $\mu$.

\begin{lemma}\label{lemma:PerfectSubtreeGroundModelCohenExt}
 Let $\mu$ be an infinite cardinal with $\mu=\mu^{{<}\mu}$ and $\dot{T}$ be an $\Add{\mu}{1}$-name for a perfect subtree of ${}^{{<}\mu}\mu$. 
 If $G$ is  $\Add{\mu}{1}$-generic over the ground model $\VV$, then there is an $x\in[\dot{T}^G]^{\VV[G]}$ with $x\notin\VV$. 
\end{lemma}

\begin{proof}
  In the following arguments, we identify $\Add{\mu}{1}$ with the set ${}^{{<}\mu}2$ ordered by reverse inclusion. We inductively  construct continuous inclusion-preserving functions $\map{i}{{}^{{<}\mu}(\mu\times 2)}{{}^{{<}\mu}2}$ and $\map{\iota}{{}^{{<}\mu}(\mu\times 2)}{{}^{{<}\mu}\mu}$ with $i(t)\Vdash\anf{\check{\iota}(\check{t}\in\dot{T})}$ for all $t\in{}^{{<}\mu}(\mu\times 2)$.

 Set $i(\emptyset)=\iota(\emptyset)$. Since both functions are continuous at nodes of limit length, we only need to discuss the successor step of the construction. 
 Assume that we already constructed $i(t)$ and $\iota(t)$ for some $t\in{}^{{<}\mu}(\mu\times 2)$. Let $D_t$ denote the dense open set of all 
extensions $p$ of $i(t)$ in $\Add{\mu}{1}$ such that $\length{p}>\length{t}$ and there are $s,t^0,t^1\in{}^{{<}\mu}\mu$ such that $p$ forces 
that $s$ is a $2$-splitting node in $\dot{T}$ above $\iota(t)$ and the pair $\langle t^0,t^1\rangle$ consists of distinct immediate successor of $s$ in 
$\dot{T}$.  For each $p\in D_t$, we fix witnesses $s_p,t^0_p,t^1_p\in{}^{{<}\mu}\mu$. Let $E_t$ denote the set of all $p\in D_t$ with 
$p\restriction\alpha\notin D_t$ for all $\alpha<\length{p}$. Then $E_t$ is a maximal antichain below $i(t)$ in $\Add{\mu}{1}$.  Fix a surjection 
$\map{f_t}{\kappa}{E_t}$. Given $\alpha<\mu$ and $k<2$, we define 
$i(t^\smallfrown\langle\alpha,k\rangle)=f_t(\alpha)^\smallfrown  k $ and $\iota(t^\smallfrown\langle\alpha,k\rangle)=t^k_{f(\alpha)}$. 
Note that the set $\Set{i(t^\smallfrown\langle\alpha,k\rangle)}{\alpha<\mu, \, k = 0,1}$ is also a maximal antichain below $i(t)$ in $\Add{\mu}{1}$.

 \begin{claim}
  If $\alpha<\mu$, then $A_\alpha=\Set{i(t)}{t\in{}^\alpha(\mu\times 2)}$ is a maximal antichain in the partial order $\Add{\mu}{1}$.
 \end{claim}
 
 \begin{proof}[Proof of the Claim]
  The above construction directly implies that $A_\alpha$ is an antichain for every $\alpha<\mu$. 
  Fix a condition $q$ in $\Add{\mu}{1}$ and $\alpha<\mu$. By the above remark, we can inductively construct continuous descending sequences $\seq{p_\beta}{\beta\leq\alpha}$ and $\seq{q_\beta}{\beta\leq\alpha}$ in $\Add{\mu}{1}$ with  $p_\beta\in A_\beta$ and  $q_\beta\leq_{\Add{\mu}{1}}p_\beta,q$ for all $\beta\leq\alpha$. Then $p_\alpha$ is an element of $A_\alpha$ compatible with $q$. 
 \end{proof}

 \begin{claim}
  If $\alpha<\mu$, then the set $\bigcup\Set{A_\beta}{\alpha<\beta<\mu}$ is dense in $\Add{\mu}{1}$. 
 \end{claim}
 
 \begin{proof}[Proof of the Claim]
  Pick a condition $q$ in $\Add{\mu}{1}$ and $\beta<\mu$ with $\alpha,\length{q}<\beta$. By the above claim, there is a $p\in A_\beta$ compatible  with $q$. Since we assumed that the support of conditions in $\Add{\mu}{1}$ is downwards-closed, we get  $p\leq_{\Add{\mu}{1}}q$. 
 \end{proof}

 \begin{claim} \label{claim:6.2.3}
  If $G$ is $\Add{\mu}{1}$-generic over $\VV$, then there is an $x\in({}^\mu(\mu\times 2))^{\VV[G]}$ such that $i(x\restriction\alpha)$ is the unique element in $A_\alpha\cap G$ for every $\alpha<\mu$. 
 \end{claim}
 
 \begin{proof}[Proof of the Claim]
  We inductively construct functions $\seq{t_\alpha}{\alpha<\mu}$ such that $t_\alpha\in{}^\alpha(\mu\times 2)$, $i(t_\alpha)\in A_\alpha\cap G$ and $t_\beta\subsetneq t_\alpha$ holds for all $\beta<\alpha<\mu$. The successor step directly follows from the above definition of the function $i$. Hence we may assume that $\alpha\in\Lim\cap\mu$ and there is a function $t\in{}^\alpha(\mu\times 2)$ such that the sequence $\seq{t\restriction\beta}{\beta<\alpha}$ satisfies the above properties. By the continuity of $i$, the condition $i(t)\in A_\alpha$ is the greatest lower bound of the descending sequence $\seq{i(t\restriction\beta)}{\beta<\alpha}$ of conditions in $G$ and therefore this condition is also an element of $G$. 
 \end{proof}

 By  Claim~\ref{claim:6.2.3} there is an $\Add{\mu}{1}$-name $\dot{y}$ for an element of $[\dot{T}]$ with the property that, whenever $G$ is $\Add{\mu}{1}$-generic over $\VV$, there is an $x\in({}^\mu(\mu\times 2))^{\VV[G]}$ with $i(x\restriction\alpha)\in A_\alpha\cap G$ and $\iota(x\restriction\alpha)\subseteq\dot{y}^G$ for all $\alpha<\mu$. Moreover, if $t\in{}^{{<}\mu}(\mu\times 2)$, $\alpha<\mu$ and $k<2$, then $i(t^\smallfrown\langle\alpha,k\rangle)\Vdash\anf{ \hspace{1.3pt} \check{t}^{\hspace{1.3pt}k}_{f_t(\alpha)}\subseteq\dot{y}}$.

 \begin{claim} \label{claim:final}
  If $G$ is $\Add{\mu}{1}$-generic over $\VV$, then $\dot{y}^G\notin\VV$. 
 \end{claim}
 
 \begin{proof}[Proof of the Claim]
  Assume towards a contradiction that there are $q \in \Add{\mu}{1}$ and $y\in{}^\mu\mu$ such that $q\Vdash\anf{\dot{y}=\check{y}}$. By the above claims, there is a function $t\in{}^{{<}\mu}(\mu\times 2)$ with $i(t)\leq_{\Add{\mu}{1}}q$. Then there is a $k<2$ with $t^k_{f_t(0)}\not\subseteq y$ and by the observation preceding this claim we get $i(t^\smallfrown\langle 0,k\rangle)\Vdash\anf{ \hspace{1.3pt}  \check{t}^{\hspace{1.3pt} k}_{f_t(0)}\subseteq\dot{y}=\check{y}}$, a contradiction. 
 \end{proof}
 
Claim~\ref{claim:final} shows that \( x = \dot{y}^G \) is as required, hence we are done. 
\end{proof}

Note that a similar argument shows that every $\mu$-Miller tree in an $\Add{\mu}{1}$-generic extension contains a cofinal branch that eventually dominates every element of ${}^\mu\mu$ from the ground model.

Remember that, given an uncountable cardinal $\delta$, an inner model $M$ has the \emph{$\delta$-cover property} (see~\cite{MR2063629}) 
if every set of ordinals of cardinality less than $\delta$ is covered by a set that is an element of $M$ and has cardinality less than $\delta$ in $M$. 
The following lemma shows that a construction used in the proof of {\cite[Theorem 7.2]{LS}} can also be used at higher cardinalities.

\begin{lemma}\label{lemma:CoverPerfectTree}
 Let $\mu$ be an infinite cardinal with $\mu=\mu^{{<}\mu}<\kappa$ and $M$ be an inner model with the $\mu^+$-cover property. 
 Define $S$ to be the subtree $({}^{{<}\mu}\mu)^M$ of ${}^{{<}\mu}\mu$ and $T$ to be the subtree $({}^{{<}\kappa}\kappa)^M$ of ${}^{{<}\kappa}\kappa$. 
 If $T$ contains a perfect subtree, then $S$ contains a perfect subtree $S_*$ with $[S_*]\subseteq M$.    
\end{lemma}

\begin{proof}
 Assume that $T$ contains a perfect subtree. Then we can construct a continuous strictly increasing  function $\map{c}{\mu+1}{\kappa}$ and a continuous inclusion-preserving injection 
 $\map{\iota}{{}^{{\leq}\mu}2}{T}$ such that $\length{\iota(s)}=c(\length{s})$ holds for all $s\in\dom{\iota}$. 
 By our assumptions, we have $\mu^+=(\mu^+)^M$, $\cof(c(\mu))^M=\mu$ and there is a set $C\in M$ with $C\subseteq c(\mu)$, $c[\mu]\subseteq C$ and $\vert{C}\vert^M=\mu$. 
 Moreover, there is a $D\in[{}^\mu 2]^\mu$ such that for every $s\in{}^\mu 2$ and every $\gamma\in C$, there is a $t\in D$ with $c(s)\restriction\gamma\subseteq t$. 
 Using an enumeration of $({}^{c(\mu)}\kappa)^M$ in $M$, we find $\bar{D}\in M$ with $\bar{D}\subseteq({}^{c(\mu)}\kappa)^M$, $D\subseteq\bar{D}$ and $\vert\bar{D}\vert^M=\mu$. 
 Work in $M$ and pick a strictly increasing sequence $\seq{\gamma_\alpha}{\alpha<\mu}$ of elements of $C$ that is cofinal in $c(\mu)$. 
 Given $\alpha<\mu$, pick an enumeration $\seq{t^\alpha_\beta}{\beta<\mu}$ of the set $\Set{t\restriction\gamma_\alpha}{t\in\bar{D}}$. 
 This allows us define a function 
\[
\Map{h}{\bar{D}}{{}^\mu\mu}{t}{h(t)},
\]
where \( h(t)(\alpha) = \min\Set{\beta<\mu}{t^\alpha_\beta\subseteq t} \)  for each \( \alpha < \mu \).

 Now work in $\VV$. Then the function $\map{i=h\circ\iota}{{}^\mu 2}{{}^\mu\mu}$ is an injection with $\ran{i}\subseteq[S]^M\subseteq[S]$. 
  Pick $s\in{}^\mu 2$, $\alpha<\mu$ and $\alpha_*<\mu$ with $\gamma_\alpha\leq c(\alpha_*)$. If $s^\prime\in{}^\mu 2$ with $s\restriction\alpha_*=s^\prime\restriction\alpha^*$, 
  then $\iota(s)\restriction\gamma_\alpha=\iota(s^\prime)\restriction\gamma_\alpha$ and we can conclude that $i(s)\restriction(\alpha+1)=i(s^\prime)\restriction(\alpha+1)$. 
 This shows that the function $i$ is continuous and we can apply Corollary~\ref{corollary:PerfectClosedTreeEmb} to derive the statement of the lemma. 
\end{proof}

We are now ready to show that the Hurewicz dichotomy fails above a regular cardinal $\mu$ after we added a single Cohen subset of $\mu$ to the ground model.

\begin{proof}[Proof of Theorem \ref{theorem:CounterexampleCohenReal}] 
 Let $\mu$ be a cardinal with $\mu=\mu^{{<}\mu}<\kappa$ and $G$ be $\Add{\mu}{1}$-generic over $\VV$. 
 Set $S=({}^{{<}\kappa}\kappa)^\VV$. 
 Since $\Add{\mu}{1}$ has cardinality less than $\kappa$, we know that $S$ is a pruned tree with $[S]=({}^\kappa\kappa)^\VV$ in $\VV[G]$. 
 In particular, $({}^\kappa\kappa)^\VV$ is a closed subset of ${}^\kappa\kappa$ in $\VV[G]$. 
 Assume towards a contradiction that the closed subset $[S]$ satisfies the Hurewicz dichotomy in $\VV[G]$. 
 Since $\VV$ has the $\mu^+$-cover property in $\VV[G]$, we can combine Lemma~\ref{lemma:PerfectSubtreeGroundModelCohenExt} and Lemma~\ref{lemma:CoverPerfectTree} to conclude that $S$ contains no perfect subtree in $\VV[G]$. 
Thus $[S]$ cannot contain a closed copy of \( \pre{\kappa}{\kappa} \) by Lemma~\ref{lemma:PerfectClosedTreeEmb}, and hence by the Hurewicz dichotomy it is (contained in) a $\K{\kappa}$ subset of ${}^\kappa\kappa$ in $\VV[G]$. In this situation, Lemma~\ref{covering kappa-compact sets in generic extensions} shows that the ground model $\VV$ contains a sequence $\seq{T_\alpha}{\alpha<\kappa}$ of pruned subtrees of ${}^{{<}\kappa}\kappa$ 
 such that $[S] =({}^\kappa\kappa)^\VV\subseteq\bigcup\Set{[T_\alpha]^{\VV[G]}}{\alpha<\kappa}$ and $[T_\alpha]$ is $\kappa$-compact in $\VV$ for all $\alpha<\kappa$. 
 But this would imply that in \( \VV \) the whole ${}^\kappa\kappa$ is a $\K{\kappa}$ subset, contradicting Fact~\ref{proposition:DichotomyUnc}.  
Note that the fact that \( [S] = (\pre{\kappa}{\kappa})^\VV \) cannot be a $\K{\kappa}$ subset of ${}^\kappa\kappa$ in $\VV[G]$ can also be proved directly without using Lemma~\ref{lemma:PerfectClosedTreeEmb}: in fact, it is enough to observe that  \( (\pre{\kappa}{\kappa})^\VV \) is unbounded in the forcing extension (because the forcing is small) and then use Lemma~\ref{lemma:bounded}. 
  \end{proof}

We close this section with an observation motivated by the above constructions.

\begin{proposition}\label{proposition:FailureWithSuperperfectTree} 
 If there is a pruned subtree $S$ of ${}^{{<}\kappa}\kappa$ such that $[S]$ does not satisfy the Hurewicz dichotomy, 
 then there is a superclosed subtree $T$ of ${}^{{<}\kappa}\kappa$ such that the closed subset $[T]$ does not satisfy the Hurewicz dichotomy. 
\end{proposition} 

\begin{proof} 
 Set $  \partial S  = \Set{t\in \pre{< \kappa}{\kappa}}{t \notin S \wedge \forall \alpha<\length{t} \, (t\restriction\alpha\in S)}$ 
and 
 \[
T  =   S \cup  \Set{t {}^\smallfrown{} 0^{(\alpha)}\in{}^{{<}\kappa}\kappa}{ s\in\partial S  \wedge \alpha < \kappa},
\]
where \( 0^{(\alpha)} \) is the sequence of length \(\alpha\) constantly equal to \( 0 \).
 Then $T$ is a superclosed tree such that $[S]\subseteq[T]$ and every branch in $[T]\setminus[S]$ is isolated in $[T]$. 
 This shows that $[T]$ is not a $\K{\kappa}$ subset of ${}^\kappa\kappa$ and every closed subset of $[T]$ without isolated points is contained in $[S]$. 
Thus since \( \pre{\kappa}{\kappa} \) has no isolated point the closed set $[T]$ cannot satisfy the Hurewicz dichotomy. 
\end{proof}


\section{Failures of the Hurewicz dichotomy in $\LL$}\label{section:FailureL}


In this section, we work in G\"odel's constructible universe $\LL$ and modify a construction of Friedman and Kulikov from~\cite{FriedmanKulikov} to obtain closed counterexamples to the Hurewicz dichotomy at every uncountable regular cardinal.

\begin{theorem}\label{theorem:CounterexampleInL}
 Assume $\VV=\LL$. If $\kappa$ is an uncountable regular cardinal, then there is a closed subset of ${}^\kappa\kappa$ that does not satisfy the Hurewicz dichotomy. 
\end{theorem}

The counterexamples constructed in the proof of this result will be of the form $[T]$ for \( T \subseteq {}^{{<}\kappa}\kappa \) as in the following proposition.

\begin{proposition}\label{proposition:ThreePropertisNoHD}
Let $T \subseteq {}^{{<}\kappa}\kappa$ be a pruned subtree with the following three properties:
\begin{enumerate-(1)}
 \item \label{proposition:ThreePropertisNoHD-1}
$T$ does not contain a perfect subtree;
 \item \label{proposition:ThreePropertisNoHD-3}
the closed set $[T]$ is $\kappa$-Baire, i.e.\ if $\seq{D_\alpha}{\alpha<\kappa}$ is a sequence of dense open subsets of $[T]$ then there is $x\in[T]$ with $x\in \bigcap_{\alpha < \kappa} D_\alpha$ (equivalently: \( T \) is not \( \kappa \)-meager, i.e.\ it is not a union of \( \kappa \)-many nowhere dense sets);
 \item \label{proposition:ThreePropertisNoHD-2}
every node in $T$ is $\kappa$-splitting.
\end{enumerate-(1)} 
Then the closed set $[T]$ does not satisfy the Hurewicz dichotomy. 
\end{proposition}

\begin{proof}
 Assume towards a contradiction that $[T]$  satisfies the Hurewicz dichotomy. 
 Since $T$ does not contain a perfect subtree by~\ref{proposition:ThreePropertisNoHD-1}, by Lemma~\ref{lemma:PerfectClosedTreeEmb} the set $[T]$ does not contain a closed subset homeomorphic to ${}^\kappa\kappa$.
 Hence $[T]$ is contained in a $K_\kappa$-subset of ${}^\kappa\kappa$. Since $[T]$ is closed, it follows that $[T]$ is equal to the union of $\kappa$-many $\kappa$-compact subsets of $[T]$. This shows that there is a sequence $\seq{T_\alpha}{\alpha<\kappa}$ of subtrees of $T$ such that $[T]=\bigcup\Set{[T_\alpha]}{\alpha<\kappa}$ and the set $[T_\alpha]$ is $\kappa$-compact for every $\alpha<\kappa$. 
 Since $[T]$ is $\kappa$-Baire by~\ref{proposition:ThreePropertisNoHD-3}, 
there is an $\alpha_*<\kappa$ such that the interior of $[T_{\alpha_*}]$ in $[T]$ is nonempty, and  thus there is a node $t_*$ in $T$ with $N_{t_*}\cap[T]\subseteq[T_{\alpha_*}]$. 
 By~\ref{proposition:ThreePropertisNoHD-2}, $t_*$ is a $\kappa$-splitting node of \( T \), and every direct successor is extended by an element of $[T]\cap N_{t_*}\subseteq[T_{\alpha_*}]$. 
 This allows us to construct an open covering of the $\kappa$-compact set $[T_{\alpha_*}]$ without a subcover of cardinality less than $\kappa$, a contradiction.  
\end{proof}

\begin{proof}[Proof of Theorem \ref{theorem:CounterexampleInL}]
 Fix an uncountable regular cardinal $\kappa$. It is enough to construct a pruned subtree \( T \subseteq \pre{<\kappa}{\kappa} \) satisfying conditions~\ref{proposition:ThreePropertisNoHD-1}--\ref{proposition:ThreePropertisNoHD-2} of Proposition~\ref{proposition:ThreePropertisNoHD}.

 If $\kappa=\mu^+$ for some cardinal $\mu$, then set $\nu = \cof(\mu)$; 
 otherwise, set $\nu=\omega$. Define a function 
 \[
\map{F}{\Set{t\in{}^{{<}\kappa}\kappa}{\length{t}\in\Lim, \, \cof(\length{t})=\nu}}{\kappa}
\]
 by setting 
\[
F(t)  =  \min\Set{\alpha<\kappa}{\length{t}<\alpha, \; \ran{t}\subseteq\alpha, \; \LL_\alpha\models\ZFC^-, \; \cof(\length{t})^{\LL_\alpha}=\nu}.
\]
 for all $t\in\dom{F}$. Finally, set 
\[
T  =  \Set{t\in{}^{{<}\kappa}\kappa}{\forall\alpha\leq\length{t} \, \left( t\restriction\alpha\in\dom{F} \Rightarrow t\restriction\alpha\in\LL_{F(t)}\right)}.
\] 
 Then $T$ is a subtree of ${}^{{<}\kappa}\kappa$ of height $\kappa$. Moreover, \( T \) is pruned: if $t\in[T]$ and $x\in{}^\kappa\kappa$ with $t\subseteq x$ and $x(\alpha)=0$ for all $\length{t}\leq\alpha<\kappa$, then $x\in[T]$. 

 \begin{claim}
  The tree $T$ does not contain a perfect subtree, and thus satisfies condition~\ref{proposition:ThreePropertisNoHD-1} of Proposition~\ref{proposition:ThreePropertisNoHD}.
 \end{claim}

 \begin{proof}[Proof of the Claim]
  Assume, towards a contradiction, that $T$ contains a perfect subtree. 
  By Lemma~\ref{lemma:PerfectClosedTreeEmb}, we can find a continuous inclusion-preserving injection $\map{\iota}{{}^{{<}\kappa}2}{T}$.  
  We inductively construct a strictly increasing continuous sequence $\seq{\lambda_\rho}{\rho\leq\nu}$ of ordinals contained in the interval $(\nu,\kappa)$ and 
  injective functions $\seq{\map{e_\rho}{\prod_{\xi<\rho}\lambda_\xi}{{}^{{<}\kappa}2}}{\rho\leq\nu}$ such that the following statements hold for all $\bar{\rho}<\rho<\nu$ and $x\in\prod_{\xi<\rho}\lambda_\xi$.
 \begin{enumerate-(a)}
  \item $e_{\bar{\rho}}(x\restriction\bar{\rho})\subseteq e_{\rho}(x)$. 

  \item $\lambda_{\bar{\rho}}\leq \length{(\iota\circ e_{\bar{\rho}})(x\restriction\bar{\rho})}<\lambda_\rho$.   

  \item $\ran{(\iota\circ e_{\bar{\rho}})(x\restriction\bar{\rho})}\subseteq\lambda_\rho$.  

  \item If $\kappa$ is an inaccessible cardinal, then $\lambda_\rho$ is a cardinal. 
 \end{enumerate-(a)}

  Given $\rho<\nu$, assume that we already constructed the sequences $\seq{\lambda_{\bar{\rho}}}{\bar{\rho}<\rho}$ and $\seq{\map{e_{\bar{\rho}}}{\prod_{\xi<\bar{\rho}}\lambda_\xi}{{}^{{<}\kappa}2}}{\bar{\rho}<\rho}$ with the above properties. If $\rho\in\Lim$, then we define $\lambda_\rho=\sup_{\bar{\rho}<\rho}\lambda_{\bar{\rho}}$ and $e_\rho(x)=\bigcup_{\bar{\rho}<\rho}e_{\bar{\rho}}(x\restriction\bar{\rho})$ for all $x\in\prod_{\xi<\rho}\lambda_\xi$. 
 Now, assume that $\rho=\bar{\rho}+1$. Since the $\mathsf{GCH}$ holds and $\rho<\nu$, the product $\prod_{\xi<\rho}\lambda_\xi$ has cardinality less than $\kappa$ and there is a $\lambda_\rho$ with $\length{(\iota\circ e_{\bar{\rho}})(x)}<\lambda_\rho$ and $\ran{(\iota\circ e_{\bar{\rho}})(x)}\subseteq\lambda_\rho$ for all $x\in\prod_{\xi<\rho}\lambda_\xi$. If $\kappa$ is inaccessible, then we find a cardinal $\lambda_\rho$ with this property. 
  Pick $y\in\prod_{\xi<\bar{\rho}}\lambda_\xi$. Then there is an $s\in{}^{{<}\kappa}2$ with $e_{\bar{\rho}}(y)\subseteq s$ and $\length{s}>\lambda_{\bar{\rho}}$. Given $\alpha<\lambda_{\bar{\rho}}$, define $e_\rho(y^\frown\langle\alpha\rangle)=s_\alpha$, where  $\map{s_\alpha}{\length{s}+\lambda_{\bar{\rho}}}{2}$ is  the unique function with $s\subseteq s_\alpha$ and $s_{\alpha}(\length{s}+\beta)=1\longleftrightarrow \alpha=\beta$ for all $\beta<\lambda_{\bar{\rho}}$. This completes the construction of the sequences.

Define $i=\iota\circ e_\nu$, $\lambda=\lambda_\nu$ and  
\begin{equation*}
  \lambda^*  =  \min\Set{\alpha<\kappa}{\lambda<\alpha, \; \LL_\alpha\models\ZFC^-, \; \cof(\lambda)^{\LL_\alpha}=\nu}. 
 \end{equation*}
Then the above construction ensures that $i$ is an injection of $\prod_{\xi<\nu}\lambda_\xi$ into $T\cap{}^\lambda\lambda$ and $(F\circ i)(x)=\lambda^*$ for all $x\in\prod_{\xi<\nu}\lambda_\xi$. In particular, $\ran{i}\subseteq\LL_{\lambda^*}$.

  If $\kappa$ is an inaccessible cardinal, then $\lambda$ is a singular strong limit cardinal of countable cofinality  and 
  \[
 \lambda^+  =  2^\lambda  =  \left\vert\prod_{n<\omega}\lambda_n\right \vert  \leq  \vert\LL_{\lambda^*}\vert  =  \lambda, 
\] a contradiciton. 
  In the other case, if there is a cardinal $\mu$ such that $\kappa=\mu^+$, then 
  \[
\kappa  =  2^\mu  =  \left\vert\prod_{\xi<\nu}\lambda_\xi \right \vert  \leq  \vert\LL_{\lambda^*}\vert  =  \mu  <  \kappa,
\] 
a contradiciton.  
 \end{proof}

 \begin{claim}
  The set $[T]$ is $\kappa$-Baire, and thus satisfies condition~\ref{proposition:ThreePropertisNoHD-3} of Proposition~\ref{proposition:ThreePropertisNoHD}. 
 \end{claim}

\begin{proof}[Proof of the Claim]
  Assume, towards a contradiction, that there exists a sequence $\vec{D}=\seq{D_\alpha}{\alpha<\kappa}$ of dense-open subsets of $[T]$ such that $\bigcap_{\alpha<\kappa}D_\alpha=\emptyset$. 
  Let $\vec{D}$ denote the $<_\LL$-least such sequence and define $\bar{D}_\alpha=\Set{s\in{}^{{<}\kappa}2}{N_s\cap[T]\subseteq D_\alpha}$ for every $\alpha<\kappa$.  
 Pick a sufficiently large $n<\omega$. The existence of $\Sigma_n$-definable $\Sigma_n$-Skolem functions in $\LL_{\kappa^+}$ (see {\cite[Section 3]{MR0309729}}) implies that 
 there is a continuous sequence $\seq{M_\alpha}{\alpha<\kappa}$  of $\Sigma_n$-elementary submodels of $\LL_{\kappa^+}$ of cardinality less than $\kappa$ 
 and a strictly increasing continuous sequence $\seq{\kappa_\alpha<\kappa}{\alpha<\kappa}$ of ordinals such that $\kappa_\alpha=\kappa\cap M_\alpha$ for all $\alpha<\kappa$ 
 and the function that sends $\alpha$ to $M_\alpha$ is $\Sigma_n$-definable in $\langle\LL_{\kappa^+},\in\rangle$. 
 Given $\alpha<\kappa$, let $\map{\pi_\alpha}{M_\alpha}{\LL_{\delta_\alpha}}$ denote the corresponding transitive collapse. 
 By the above choices, we have $\kappa,\nu,F,T\in M_0$ and hence 
 \[
T\cap {}^{{<}\kappa_\alpha}2, \,  \seq{\bar{D}_\beta\cap {}^{{<}\kappa_\alpha}2}{\beta<\kappa_\alpha}  \in  \LL_{\delta_\alpha}
\] 
 for every $\alpha<\kappa$. Moreover, $\kappa_\alpha$ is a regular cardinal in $\LL_{\delta_\alpha}$ and hence we have $\alpha\leq\kappa_\alpha<\delta_\alpha<F(t)$ for every $\map{t}{\kappa_\alpha}{\kappa}$.

 We inductively construct a strictly increasing sequence $\seq{t_\alpha}{\alpha<\kappa}$ of nodes in $T$ such that the following statements hold for all $\alpha<\kappa$. 
 \begin{enumerate-(a)}
  \item $t_0=\emptyset$. 

  \item $t_{\alpha+1}$ is the $<_\LL$-least element $t$ of $\bar{D}_\alpha$ with $\length{t}\geq\kappa_\alpha$.  

  \item If $\alpha\in\Lim$, then $t_\alpha=\bigcup\Set{t_{\bar{\alpha}}}{\bar{\alpha}<\alpha}$. 
 \end{enumerate-(a)}

  Let $\alpha<\kappa$ and assume that we already constructed a sequence $\vec{t}=\seq{t_{\bar{\alpha}}}{\bar{\alpha}<\alpha}$ with the above properties. 
  Since $N_t\cap[T]\neq\emptyset$ for every $t\in T$, we may assume that $\alpha$ is a limit ordinal of cofinality $\nu$.  
  By the above assumptions, we have  $\alpha\leq\kappa_\alpha=\kappa\cap M_\alpha$, the function 
  \[
\Map{g}{\kappa}{\kappa}{\beta}{\kappa_\beta}
\] 
is definable in $M_\alpha$ and the sequence $\seq{\bar{D}_\beta}{\beta<\kappa}$ is an element of $M_\alpha$. This allows us to  conclude that the sequence $\vec{t}$ and the function $t_\alpha=\bigcup_{\bar{\alpha}<\alpha}t_{\bar{\alpha}}$ are both definable in $\langle\LL_{\delta_\alpha},\in\rangle$.   In particular, $t_\alpha$ is an element of $\LL_{t_\alpha}$ and a node in $T$.  

  If we define $x=\bigcup\Set{t_\alpha}{\alpha<\kappa}$, then $x\in[T]$ with $x\in D_\alpha$ for every $\alpha<\kappa$, a contradiction. 
 \end{proof}

By the definition of $T$, if $t$ is a node in $T$, then $t^\smallfrown \alpha \in T$ for every $\alpha<\kappa$. 
In particular, every node in $T$ is $\kappa$-splitting and thus \( T \) satisfies also condition~\ref{proposition:ThreePropertisNoHD-2} of Proposition~\ref{proposition:ThreePropertisNoHD}. This concludes our proof.
\end{proof}


\section{Concluding remarks and open questions}

In~\cite{PS}, the third author shows that if $\lambda>\kappa$ is an inaccessible cardinal and $G$ is $\mathrm{Col}(\kappa,\lambda)$-generic over $\VV$, 
then every subset of ${}^{\kappa}\kappa$ definable from ordinals and subsets of $\kappa$ has the \( \kappa \)-perfect set property in $\VV[G]$. 
In particular, if $\kappa$ is not weakly compact in $\VV[G]$ then every subset of ${}^{\kappa}\kappa$ definable from ordinals and subsets of $\kappa$ satisfies the Hurewicz dichotomy. 
In the light of the results of this paper, this consistency result leads to a number of interesting questions. First, we ask whether the use of an inaccessible cardinal is necessary.

\begin{question} 
 If the Hurewicz dichotomy holds for $\mathbf{\Pi}^1_1$ subsets of ${}^{\kappa}\kappa$ (i.e.\ complements of $\mathbf{\Sigma}^1_1$ subsets), is there an inner model with an inaccessible cardinal? 
\end{question}

Note that in the classical case, we can construct a model in which all $\mathbf{\Sigma}^1_2$ subsets of ${}^\omega\omega$ satisfy the Hurewicz dichotomy without large cardinals by iterating Hechler forcing $\omega_1$-many times. 
Next, we ask whether the conclusion of Theorem~\ref{theorem:HDinCollapse} extends to all definable subsets of ${}^\kappa\kappa$.

\begin{question} 
 If $\lambda>\kappa$ is an inaccessible cardinal and $G$ is $\mathrm{Col}(\kappa,\lambda)$-generic over $\VV$ such that $\kappa$ is weakly compact in $\VV[G]$, 
 does every subset of ${}^\kappa\kappa$ in $\VV[G]$ that is definable from ordinals and subsets of $\kappa$ satisfy the Hurewicz dichotomy?  
\end{question}

As we already observed, at non-weakly compact cardinals the Hurewicz dichotomy is a consequence of the \( \kappa \)-perfect set property, but it is not clear whether this implication might fail at weakly compact cardinals.

\begin{question} 
 Is it consistent that there is a weakly compact cardinal $\kappa$ such that all $\mathbf{\Sigma}^1_1$ subsets of ${}^\kappa\kappa$ have the \( \kappa \)-perfect set property but there is a closed (or even just \( \mathbf{\Sigma}^1_1 \)) subset of ${}^\kappa\kappa$ that does not satisfy the Hurewicz dichotomy? 
\end{question}

By Theorem~\ref{theorem:HDimpliesMeasurability}, if $\kappa$ is weakly compact then the Hurewicz dichotomy for $\mathbf{\Sigma}^1_1$ subsets of ${}^{\kappa}\kappa$ implies that these subsets are $\kappa$-Miller measurable.

\begin{question}
 If $\kappa$ is not a weakly compact cardinal, does the assumption that all $\mathbf{\Sigma}^1_1$ subsets satisfy the Hurewicz dichotomy imply that these subsets are $\kappa$-Miller measurable?
\end{question}

We also ask about the converse. 

\begin{question} 
 Is it consistent that all  $\mathbf{\Sigma}^1_1$ subsets of ${}^\kappa\kappa$ are $\kappa$-Miller measurable and there is a $\mathbf{\Sigma}^1_1$ subset that does not satisfy the Hurewicz dichotomy?  Can such a $\kappa$ be weakly compact? 
\end{question}

It was mentioned in Section~\ref{section:intro} that the strong Hurewicz dichotomy for $\mathbf{\Sigma}^1_1$ subsets of ${}^{\kappa}\kappa$ provably fails for weakly compact $\kappa$. In the light of \cite[Theorem 21.18]{MR1321597}, it is also interesting to consider the following weakening of Definition~\ref{definition:StrongHD}.

\begin{definition}
 Given an infinite cardinal $\kappa$, we say that a subset $A$ of ${}^\kappa\kappa$ satisfies the \emph{relativized strong Hurewicz dichotomy} 
 if either $A$ is a \( \K{\kappa} \) subset of ${}^\kappa\kappa$ or $A$ contains a \emph{relatively closed} (that is, a closed-in-\( A \)) subset homeomorphic to ${}^\kappa\kappa$. 
\end{definition}

Note that the constructions of Sections~\ref{section:FailureCohen} and~\ref{section:FailureL} also show that consistently there are closed counterexamples to the relativized strong Hurewicz dichotomy: in fact, for closed sets the latter is clearly equivalent to the strong Hurewicz dichotomy. Moreover, Theorem~\ref{theorem:ConsNonWC} already shows that if $\kappa$ is not weakly compact, then the relativized strong Hurewicz dichotomy can be forced to hold for all $\mathbf{\Sigma}^1_1$ subsets of ${}^{\kappa}\kappa$.

\begin{question} 
 Is it consistent that for some weakly compact $\kappa$ the relativized strong Hurewicz dichotomy holds for all $\mathbf{\Sigma}^1_1$ subsets of ${}^{\kappa}\kappa$? 
\end{question}


\end{document}